\documentclass[12pt]{article}

\usepackage{amsfonts}
\usepackage{amssymb}
\usepackage{amsthm}
\usepackage{amsmath}
\usepackage{amscd}
\usepackage{mathrsfs}
\usepackage{enumerate}
\usepackage{ulem}
\usepackage{fontenc}
\usepackage[all]{xy}
\usepackage{array}
\usepackage[english]{babel}

\def\CC {{\mathbb C}}     
\def\NN {{\mathbb N}}     
\def\PP {{\mathbb P}}     
\def\RR {{\mathbb R}}     
\def\ZZ {{\mathbb Z}}     

\def\lo  {\longmapsto}

\def\Lw  {\Longrightarrow}
\def\lw  {\longrightarrow}
\def\mc {\mathcal}
\def\mk {\mathfrak}

\def\ol  {\overline}

\def\rw  {\rightarrow}
\def\tst {\Longleftrightarrow}

\def\wt  {\widetilde}

\newcommand{\sco}[1]{\left(#1\right)}

\newtheorem{theorem}{Theorem}[section]
\newtheorem{lemma}[theorem]{Lemma}

\newtheorem{prop}[theorem]{Proposition}

\newtheorem{coro}[theorem]{Corollary}
\newtheorem{zab}{Remark}[section]
\newtheorem{df}{Definition}[section]

\newtheorem{ex}{Example}[section]

\begin{document}

\begin{center}

{\LARGE Embeddings of semisimple complex Lie groups and cohomological components of modules}\\

\vspace{0.5cm}
{\large\textsc{Valdemar V. Tsanov}} \footnote{\it Present address: Ruhr-Universit\"at Bochum, Fakult\"at f\"ur Mathematik, NA 4/76, Bochum 44780, Deutschland. Email: Valdemar.Tsanov@rub.de}
\vspace{1cm}

{\it Queen's University, Department of Mathematics and Statistics, Kingston, Ontatio, K7L 3N6, Canada}\\
{\it Email: valdemar.tsanov@gmail.com}\\
\end{center}

\begin{abstract}
Let $G\hookrightarrow\wt{G}$ be an embedding of semisimple complex Lie groups, $B\subset\wt{B}$ a pair of nested Borel subgroups and $G/B\hookrightarrow\wt{G}/\wt{B}$ the associated embedding of flag manifolds. Let $\wt{\mc O}(\wt{\lambda})$ be an equivariant invertible sheaf on $\wt{G}/\wt{B}$ and ${\mc O}(\lambda)$ be its restriction to $G/B$. Consider the $G$-equivariant pullback
$$
\pi^{\wt{\lambda}} : H(\wt{G}/\wt{B},\wt{\mc O}(\wt{\lambda})) \lw H(G/B,{\mc O}(\lambda)) \;.
$$
The Borel-Weil-Bott theorem and Schur's lemma imply that $\pi^{\wt{\lambda}}$ is either surjective or zero. If $\pi^{\wt{\lambda}}$ is nonzero, the image of the dual map $(\pi^{\wt{\lambda}})^*$ is a $G$-irreducible component in a $\wt{G}$-irreducible module, called a cohomological component.

We establish a necessary and sufficient condition for nonvanishing of $\pi^{\wt{\lambda}}$. Also, we prove a theorem on the structure of the set of pairs of dominant weights $(\mu,\wt{\mu})$ with $V(\mu)\subset\wt{V}(\wt{\mu})$ cohomological. Here $V(\mu)$ and $\wt{V}(\wt{\mu})$ denote the respective highest weight modules. Simplified specializations are formulated for regular and diagonal embeddings. In particular, we give an alternative proof of a recent theorem of Dimitrov and Roth. Beyond the regular and diagonal cases, we study equivariantly embedded rational curves and we also show that the generators of the algebra of ad-invariant polynomials on a semisimple Lie algebra can be obtained as cohomological components. Our methods rely on Kostant's theory of Lie algebra cohomology.
\end{abstract}

\section*{Introduction}

The classical Borel-Weil-Bott theorem provides realizations for the irreducible finite dimensional modules of a semisimple complex Lie group $G$ as cohomology spaces associated to irreducible homogeneous vector bundles on the flag manifolds of $G$. We address the ``functorial'' behavior of this construction, i.e. we consider equivariant embeddings of flag manifolds and the resulting pullbacks of cohomology. In this paper, we concentrate on the case of complete flag manifolds and cohomology of line bundles. The study of vector bundles on partial flag manifolds will be carried out in a future work.

Let us start by introducing some basic notation and then proceed to discuss the objective of this paper. We assume that $G$ is connected and simply connected. Let $B\subset G$ be a Borel subgroup and $H\subset B$ be a Cartan subgroup. Let ${\mc P}$, $\Delta$, ${\mc W}$ denote respectively the weight lattice, root system and Weyl group of $G$ with respect to $H$. Let ${\mc P}^+$, $\Delta^+$, $\Delta^-$, $\Pi$ denote respectively the sets of dominant weights, positive roots, negative roots and simple roots of $G$ with respect to $B$. Let $\rho=\frac{1}{2}\langle\Delta^+\rangle$, where $\langle\Phi\rangle$ denotes the sum of the elements of a finite subset $\Phi\subset{\mc P}$. To each element $w\in{\mc W}$ is associated its inversion set $\Phi_w=\Delta^+\cap w^{-1}(\Delta^-)$, which in fact characterizes the element. The length $l(w)$, i.e. the number of simple reflections in a minimal expression for $w$, equals the cardinality of $\Phi_w$. Besides the linear action of ${\mc W}$ on ${\mc P}$, denoted by $w(\lambda)$, we consider the affine action defined by $w\cdot\lambda=w(\lambda+\rho)-\rho$. We record for future reference the following formulae relating the two actions of ${\mc W}$ on ${\mc P}$:
\begin{gather}\label{For w^-1(w.lambda)}
w^{-1}(w\cdot\lambda)=\lambda-w^{-1}\cdot 0 \quad,\quad w^{-1}\cdot 0=-\langle\Phi_{w}\rangle \;.
\end{gather}
A weight $\lambda\in{\mc P}$ is called regular if it is congruent to a dominant weight under the affine action of ${\mc W}$. In such a case, the relevant Weyl group element, which is unique, is denoted by $w_\lambda$, so that $w_\lambda\cdot\lambda\in{\mc P}^+$. A weight which is not regular is called singular. For regular weights, we define length as $l(\lambda)=l(w_\lambda)$. For singular weights the length function is not defined; whenever we write $l(\lambda)$ we make the implicit assumption that $\lambda$ is regular. If $V$ is a $G$-module, its weight spaces are denoted by $V^\lambda$ for $\lambda\in{\mc P}$; a generic weight vector is denoted by $v^\lambda$ (weight vectors are assumed to be non-zero). For $\lambda\in{\mc P}^+$, let $V(\lambda)$ denote an irreducible $G$-module with highest weight $\lambda$; the weights of the form $w(\lambda)$, $w\in{\mc W}$ are called the extreme weights of $V(\lambda)$; the corresponding weight vectors $v^{w(\lambda)}$ are called extreme weight vectors. Dual spaces and maps are denoted by a subscript $^*$. We use the following common convention: lower case German letters denote Lie algebras of Lie groups denoted by the corresponding capital Italic letters. The universal enveloping algebra of a Lie algebra ${\mk a}$ is denoted by ${\mk U}(\mk a)$. When discussing cohomology, we denote by $H(--)$ the total cohomology group $\oplus_q H^q(--)$.

Let $X=G/B$ denote the flag manifold of $G$. The $G$-equivariant line bundles on $X$ are parametrized by the characters of $B$, which are in turn parametrized by ${\mc P}$. Let $\CC_\lambda$ denote a one dimensional $B$-module with character $\lambda\in{\mc P}$. Let ${\mc O}(\lambda)$ denote the sheaf of local holomorphic sections of the line bundle ${\mc L}_{-\lambda}=G\times_B\CC_{-\lambda}$. We can now formulate

\begin{theorem}{\rm (Borel-Weil-Bott, \cite{Bott})} Let $\lambda\in{\mc P}$. Then
$$
H^{q}(X,{\mc O}(\lambda)) \cong \begin{cases} V(w_\lambda\cdot\lambda)^* \quad &,\; if\;\; l(\lambda)=q \\
                                               0 &,\; otherwise \;.
                                \end{cases}
$$
\end{theorem}

Suppose that $G$ is embedded as a subgroup into another semisimple complex Lie group $\wt{G}$. The group $B$, being a solvable subgroup of $\wt{G}$, can be embedded into a Borel subgroup $\wt{B}\subset\wt{G}$. Put $\wt{X}=\wt{G}/\wt{B}$. For convenience, we also fix Cartan subgroups $H\subset B$ and $\wt{H}\subset\wt{B}$ with $H=B\cap\wt{H}$. Notice however, that many of the concepts introduced below depend only on the choice of Borel subgroups. Let $\iota^*:\wt{\mc P}\rw{\mc P}$ denote the restriction of weights. There is an equivariant embedding of flag manifolds
$$
\varphi: X \hookrightarrow \wt{X} \;.
$$
Given an invertible sheaf $\wt{\mc O}(\wt{\lambda})$ on $\wt{X}$, it restricts to an invertible sheaf ${\mc O}(\lambda)$ on $X$, with $\lambda=\iota^*(\wt{\lambda})$. There is a $G$-equivariant pullback map
$$
\pi^{\wt{\lambda}}: H(\wt{X},\wt{\mc O}(\wt{\lambda})) \lw H(X,{\mc O}(\lambda)) \;.
$$
By the Borel-Weil-Bott theorem, the domain space of $\pi^{\wt{\lambda}}$ is an irreducible $\wt{G}$-module, $\wt{V}(\wt{w}\cdot\wt{\lambda})^*$, and the target space is an irreducible $G$-module, $V(w\cdot\lambda)^*$. By Schur's lemma, $\pi^{\wt{\lambda}}$ is either surjective or zero. When $\pi^{\wt{\lambda}}$ is nonzero, its dual is a $G$-monomorphism $(\pi^{\wt{\lambda}})^*:V(w\cdot\lambda)\hookrightarrow\wt{V}(\wt{w}\cdot\wt{\lambda})$. This brings us to the central notion considered in this paper.

\begin{df} Given an irreducible $\wt{G}$-module $\wt{V}$, an irreducible $G$-submodule $V\subset\wt{V}$ is called a {\rm cohomological component}, if it can be realized as the image of $(\pi^{\wt{\lambda}})^*$ for an appropriate $\wt{\lambda}$. The {\rm set of cohomological pairs of dominant weights} is defined as
\begin{gather}\label{For C in intro}
{\mc C} = {\mc C}(\varphi) = \{ (\mu,\wt{\mu})\in{\mc P}^+\times\wt{\mc P}^+ \;:\; V(\mu)\subset\wt{V}(\wt{\mu}) \; {\rm cohomological} \} \;.
\end{gather}
\end{df}

The construction of cohomological components could be termed as ``geometric branching'' of representations, versus the general ``algebraic branching''. Here are three problems arising naturally after the above definition.\\

\noindent{\bf Problem I}: {\it Find a criterion for nonvanishing of the pullback $\pi^{\wt{\lambda}}$.}\\
\noindent{\bf Problem II}: {\it Characterize the cohomological components of a given irreducible $\wt{G}$-module.}\\
\noindent{\bf Problem III}: {\it Describe the set ${\mc C}$ of cohomological pairs of dominant weights.}\\

The goal of this paper is to address these problems. We present a complete solution to Problem I. The resulting criterion actually gives a solution to Problem II as well. Problem III is addressed in Section \ref{Sec PropCohCom}. A complete solution is not given, but some important structural properties of the set ${\mc C}$ are established, reducing Problem III to a saturation problem for some finitely generated submonoids of ${\mc P}^+\times\wt{\mc P}^+$.\\

The initial motivation for this study came after the work of Dimitrov and Roth, \cite{Dimitrov-Roth short}, \cite{Dimitrov-Roth long}, where the above problems were posed and solved for diagonal embeddings. It is in this work that the notion of a cohomological component was originally introduced. In the case of a diagonal embedding, $X\hookrightarrow X\times X$, K\"unneth's formula allows an interpretation of the pullback as a cup product map
$$
\pi^{(\lambda_1,\lambda_2)}: H^{q_1}(X,{\mc O}(\lambda_1)) \otimes H^{q_2}(X,{\mc O}(\lambda_2)) \lw H^{q_1+q_2}(X,{\mc O}(\lambda)) \;,
$$
where $\lambda=\lambda_1+\lambda_2$. When the dual map $(\pi^{(\lambda_1,\lambda_2)})^*$ is nonzero, its image is an irreducible component $V(w_\lambda\cdot\lambda)$ of the tensor product $V(w_{\lambda_1}\cdot\lambda_1)\otimes V(w_{\lambda_2}\cdot\lambda_2)$. This makes the diagonal case particularly interesting, as it relates to Littlewood-Richardson theory. Dimitrov and Roth found the following solution to Problem I.

\begin{theorem}{\rm (Dimitrov-Roth, \cite{Dimitrov-Roth long})}
The cup product map $\pi^{(\lambda_1,\lambda_2)}$ is nonzero if and only if
\begin{gather}\label{For Cond D-R}
\Phi_{w_{_{\lambda_1}}} \sqcup \Phi_{w_{_{\lambda_2}}} = \Phi_{w_{_\lambda}} \;.
\end{gather}
\end{theorem}

The methods employed in \cite{Dimitrov-Roth long} for the proof of the above criterion for nonvanishing are algebro-geometric. However, condition (\ref{For Cond D-R}) is Lie theoretic in nature and it seemed desirable to have a proof and ``explanation'' for this criterion in Lie theoretic terms. The necessary framework is readily available in Kostant's fundamental paper on Lie algebra cohomology, \cite{Kostant}, and turns out to be sufficiently flexible to allow treatment of the general case. Kostant's theorem (Th. 5.14 in \cite{Kostant}, presented here as Th. \ref{Theo Kostant's}) is a Lie-algebraic counterpart to the Borel-Weil-Bott theorem and has a particularly strong feature: explicit harmonic representatives.

In Theorem \ref{Theo Necessity} of the present paper we provide a complete solution of problem I. The criterion can be formulated briefly, using Kostant's harmonics, as follows:\\

{\it The pullback $\pi^{\wt{\lambda}}$ is non-zero if and only if it can be realized on harmonic representatives.}\\

Kostant's harmonics are very explicit and the above criterion can be formulated in concrete terms. This is the content of Theorem \ref{Theo Necessity}, where we present two conditions, (i) and (ii), which are necessary and, together, sufficient for the nonvanishing of $\pi^{\wt{\lambda}}$. The first one is an analogue of (\ref{For Cond D-R}) above, refers only to the Weyl group elements $\wt{w}$ and $w$, and is easy to verify for a given pair $\wt{w},w$. The second one refers to the $G$-module structure of the cohomology group $H(\wt{X},\wt{\mc O}(\wt{w}\cdot\wt{\lambda}))$, and is considerably more difficult to verify. The good news is that the first condition is quite restrictive already and, in several important cases, automatically implies the second. However, we show that this implication does not hold in general. A simple counter-example is supplied via the adjoint representation of $SL_2$ (see Theorem \ref{Theo principal}).

Condition (ii) of Theorem \ref{Theo Necessity} provides an answer to problem II: a component $V(\mu)\subset\wt{V}(\wt{\mu})$ is cohomological if and only if there exists a pair of Weyl group elements $w,\wt{w}$ satisfying (i), for which the extreme weight vector $v^{w^{-1}(\mu)}\in V(\mu)$ is not orthogonal to the extreme weight vector $\wt{v}^{\wt{w}^{-1}(\wt{\mu})}$ with respect to any $\wt{K}$-invariant Hermitian form on $\wt{V}(\wt{\mu})$. Here $\wt{K}$ denotes a maximal compact subgroup of $\wt{G}$ containing a maximal compact subgroup of $G$. An alternative formulation of (ii), with more geometric flavor, is given in Theorem \ref{Theo Necessity}$'$. Although complete, this solution to problem II is not completely satisfactory, since the relevant property can be hard to check. As already mentioned, a good characterization of cohomological components is known for special classes of embeddings. Dimitrov and Roth proved in \cite{Dimitrov-Roth long} that, in the case of diagonal embeddings, the generic cohomological components are generalized PRV components of stable multiplicity 1 and vice versa. In Theorem \ref{Theo Regular embedding} we show that, for regular embeddings, the cohomological components are those generated by common highest weight vectors for $B$ and $\wt{B}$ (i.e. the highest components). In Theorem \ref{Theo principal} we give an answer to problem II for the case of principal rational curves. It turns out that the list is very small in this case: the cohomological components are either highest components obtained via pullbacks of global sections, or trivial modules obtained in degree 1 of cohomology.

In Theorem \ref{Theo C_w,wt(w) monoid} we establish some structural properties of the set ${\mc C}$, making a first step towards a solution of problem III. For a pair of Weyl group elements $w,\wt{w}$ satisfying condition (i) of Theorem \ref{Theo Necessity}, we denote
$$
{\mc D}_{w,\wt{w}} = \{ (\mu,\wt{\mu})\in{\mc P}^+\times\wt{\mc P}^+ \;:\; \wt{w}^{-1}\cdot\wt{\mu} \stackrel{\iota^*}{\lo} w^{-1}\cdot\mu \}\;.
$$
This set is a submonoid of the dominant monoid ${\mc P}^+\times\wt{\mc P}^+$. It consists, in view of Theorem \ref{Theo Necessity}, of all pairs $(\mu,\wt{\mu})$ such that $V(\mu)$ could be a cohomological component of $\wt{V}(\wt{\mu})$ associated with the given pair $w,\wt{w}$. The set of cohomological pairs introduced above can be written as
$$
{\mc C} = \bigcup {\mc C}_{w,\wt{w}} \;, \quad {\rm where} \quad {\mc C}_{w,\wt{w}} = {\mc C} \cap {\mc D}_{w,\wt{w}}
$$
and the union is along all pairs $w,\wt{w}$ satisfying (i). Theorem \ref{Theo C_w,wt(w) monoid} asserts that $C_{w,\wt{w}}$ is a submonoid of ${\mc D}_{w,\wt{w}}$ and, furthermore, there exists a positive integer $k$ such that for any $(\mu,\wt{\mu})\in{\mc D}_{w,\wt{w}}$ we have $(k\mu,k\wt{\mu})\in{\mc C}_{w,\wt{w}}$. Various situations occur: the complement ${\mc D}_{w,\wt{w}}\setminus{\mc C}_{w,\wt{w}}$ can be finite, infinite, or empty. A variety of examples is supplied by the results of Section \ref{Sec Invariants}.

The paper consists of three parts. Part 1 summarizes the necessary background. The basic goal here is to state Kostant's theorem on Lie algebra cohomology (Theorem \ref{Theo Kostant's} in the text), and sketch its connection to the Borel-Weil-Bott theorem. Part 2 contains a detailed setting of the situation we are interested in, and the main results. In Section \ref{Sec Pullbacks and CohCom} we prove Theorem \ref{Prop Translate}, which asserts that the nonvanishing of $\pi^{\wt{\lambda}}$ can be studied in either of the three cohomology theories: sheaf cohomology, Dolbeault cohomology, and Lie algebra cohomology. Section \ref{Sec Criterion} contains the criterion for nonvanishing of $\pi^{\wt{\lambda}}$ (Theorem \ref{Theo Necessity}) and some related results and remarks. Section \ref{Sec PropCohCom} contains the results on the structure of ${\mc C}$, notably Theorem \ref{Theo C_w,wt(w) monoid}. Part 3 contains applications, specializations and examples. Sections \ref{Sec Regular} and \ref{Sec Diagonal} are devoted to two important special classes of embeddings. Section \ref{Sec Regular} deals with regular embeddings; the results here are new and although one finds nothing exotic or unexpected, this case is arguably important for the general picture. Section \ref{Sec Diagonal} deals with diagonal embeddings and contains an alternative proof of the nonvanishing criterion of Dimitrov and Roth. In Section \ref{Sec RationalCurves} we study homogeneous rational curves in flag manifolds. We give a complete description of the cohomological components or the case of curves associated with principal $\mk{sl}_2$ subalgebras. With this result and the result on regular embeddings in hand, we can access a larger class of rational curves via factorizations of embeddings. Finally, in Section \ref{Sec Invariants} we study cohomological components arising from the adjoint representation of a semisimple complex Lie algebra ${\mk g}$ and show that a full set of generators for the algebra of ad-invariant polynomials on ${\mk g}$ can be obtained as cohomological components.

\section{Preliminaries}\label{Sec Preliminar}

\subsection{Semisimple Lie algebras and Lie groups}\label{Par SSLieAlg}

Let ${\mk g}$ be a semisimple complex Lie algebra and $G$ be the associated connected, simply connected, complex Lie group. Let $\kappa$ denote the Killing form on ${\mk g}$. We will use the same notation for the restrictions of $\kappa$ to subalgebras, some extensions, and the induced bilinear forms on dual spaces. In particular, $\kappa$ extends to a nondegenerate bilinear form on the Grassmann algebra $\Lambda{\mk g}$, where the degree components are declared to be mutually orthogonal and on pure tensors we have $\kappa(x_1\wedge...\wedge x_q,y_1\wedge...\wedge y_q)=\det\kappa(x_j,y_k)$. Thus $\kappa$ yields an algebra isomorphism $\Lambda{\mk g}\cong\Lambda{\mk g}^*$.

Let $K\subset G$ be a maximal compact subgroup and ${\mk k}\subset{\mk g}$ be the corresponding compact real form, so that ${\mk g}={\mk k}\oplus i{\mk k}$. The restriction of $\kappa$ to ${\mk k}$ is negative definite and coincides with the Killing form of $\mk{k}$. Put ${\mk q}=i{\mk k}$, so that ${\mk g}={\mk q}\oplus i{\mk q}$, and $\kappa$ is positive definite on ${\mk q}$. A conjugation in ${\mk g}$ is defined by
$$
\ol{x+iy}=x-iy \quad,\quad x,y\in{\mk q} \;.
$$
In turn, this conjugation defines a positive definite Hermitian form $\{.,.\}$ on ${\mk g}$ by
$$
\{ x,y \} = \kappa(x,\ol{y}) \quad,\quad x,y\in{\mk g} \;.
$$
Since $\Lambda{\mk g}=\Lambda_{\RR}{\mk q}\oplus_{\RR} i \Lambda_{\RR}{\mk q}$, both the conjugation and the Hermitian form extend to $\Lambda{\mk g}$. The following proposition summarizes some facts from \cite{Kostant}, $\S$ 3.

\begin{prop}\label{Prop ol(a)=a*} Let ${\mk a}$ be any subalgebra of ${\mk g}$.

{\rm (i)} Then $\ol{\mk a}=\{\ol{a}\in{\mk g} : a\in{\mk a}\}$ is a subalgebra of ${\mk g}$.

{\rm (ii)} The map $\ol{\mk a}\rw{\mk a}^*$, $x\mapsto \xi_x$ defined by $\xi_x(a)=\kappa(a,x)$, for $a\in{\mk a}$, is an isomorphism of vector spaces, which extends naturally to a graded isomorphism of the Grassmann algebras $\Lambda\ol{\mk a}\cong \Lambda{\mk a}^*$.

{\rm (iii)} The pullback of $\{.,.\}$ along $\Lambda{\mk a}^* \cong \Lambda\ol{\mk a} \hookrightarrow \Lambda{\mk g}$ is a positive definite Hermitian form on $\Lambda{\mk a}^*$.
\end{prop}

Recall that in the introduction we have fixed a Cartan and a Borel subgroups of $G$. We may, and do, assume that $T=K\cap H$ is a maximal torus in $K$. We have the corresponding subalgebras ${\mk h}\subset{\mk b}\subset{\mk g}$ and a root space decomposition
$$
{\mk g}={\mk h}\oplus(\oplus_{\alpha\in\Delta}{\mk g}^\alpha)
$$
The restriction of $\kappa$ to ${\mk h}$ provides an isomorphism between ${\mk h}^*$ and ${\mk h}$, denoted by $\mu\mapsto h_\mu$. In particular, we have the coroots $h_{\alpha}$, for $\alpha\in\Delta$. We fix root vectors $e_{\alpha}\in{\mk g}^\alpha$, for $\alpha\in\Delta$, satisfying $[e_\alpha,e_{-\alpha}]=h_\alpha$ and
\begin{equation}\label{For OrthRel e_alf}
\kappa(e_\alpha,e_\beta)=\begin{cases} 1 \;\;,\; {\rm if}\;\; \alpha=-\beta\;,\\
                                 0 \;\;,\; {\rm if}\;\; \alpha\ne-\beta\;.
                   \end{cases}
\end{equation}
One verifies immediately that
\begin{gather}\label{For ol(e_alp)=e_-alp}
\ol{e_\alpha}=e_{-\alpha}\;.
\end{gather}
Denote ${\mk n}=[{\mk b},{\mk b}]$; this is the nil-radical of ${\mk b}$. Also, set ${\mk n}^-=\ol{\mk n}$. Then
$$
{\mk n}={\mk n}^{+}=\oplus_{\alpha\in\Delta^{+}}{\mk g}^{\alpha} \quad,\quad {\mk n}^{-}=\oplus_{\alpha\in\Delta^{-}}{\mk g}^{\alpha} \;.
$$

Following Dynkin, we call a subalgebra ${\mk a}$ of ${\mk g}$ regular with respect to the Cartan subalgebra ${\mk h}$, if ${\mk a}$ is also an ${\mk h}$-submodule of ${\mk g}$ with respect to the adjoint action. In such a case, we set $\Delta(\mk a)=\{a\in\Delta \vert {\mk g}^{\alpha}\subset{\mk a}\}$. A subalgebra ${\mk a}\subset{\mk g}$ ia called regular, if it is regular with respect to some Cartan subalgebra. In particular, any subalgebra of ${\mk g}$ containing ${\mk h}$ is a regular subalgebra. The following proposition is straightforward to verify.

\begin{prop}\label{Prop ol(a)=a* regular} Let ${\mk a}$ be a regular subalgebra of ${\mk g}$ with respect to ${\mk h}$.

{\rm (i)} $\Delta(\ol{\mk a})=-\Delta(\mk a)$.

{\rm (ii)} The isomorphism $\ol{\mk a}\cong{\mk a}^*$ and the monomorphism $\Lambda{\mk a}^*\hookrightarrow\Lambda{\mk g}$ defined in Proposition \ref{Prop ol(a)=a*} are morphisms of ${\mk h}$-modules.

{\rm (iii)} If ${\mk a}\cap{\mk h}=0$, then $\{ e_{\Phi}=\wedge_{\alpha\in\Phi}e_{\alpha} \; \vert \; \Phi\subset\Delta(\mk a) \}$ is a basis of weight vectors for the ${\mk h}$-module $\Lambda{\mk a}$.
\end{prop}

\subsection{Flag manifolds and homogeneous vector bundles}\label{Sec Flags and Bundles}

Recall that $X=G/B$ is the flag manifold of $G$. Put $x_o=eB\in X$. A holomorphic vector bundle ${\mc E}\rw X$ is said to be homogeneous (or more precisely $G$-homogeneous), if the total space ${\mc E}$ caries a $G$-action such that the map ${\mc E}\rw X$ is $G$-equivariant and for any $g\in G$, $x\in X$ the resulting map between the fibres $E_x\rw E_{gx}$ is linear. Let $E_o$ be the fibre over $x_o$. Then $B$ acts linearly on $E_o$ and we have
$$
{\mc E}\cong G\times_B E_o = (G\times E_o)/((gb,v)\sim(g,bv))\;.
$$

The action of the compact group $K$ on $X$ is also transitive (cf. \cite{Wallach HarmHomSp}). The intersection $T=K\cap B=K\cap H$ is a maximal abelian subgroup of $K$ and the flag manifold can be written as $X=K/T$. Let $p:K\rw X$, $g\mapsto gx_o$ be the canonical map. For $g\in K$ let $p_g^\CC:T_g^\CC K\rw T_{gx_o}^\CC X$ be the tangent map of complexified tangent spaces. By definition ${\mk g}=T_e^\CC K$. Thus
$$
{\rm Ker}(p_e^\CC)={\mk h} \quad,\quad {\rm Im}(p_e^\CC)=T_{x_o}^\CC X\cong{\mk n}^-\oplus{\mk n} \;.
$$
Let $g\in K$ and let ${\mk g}$ be identified with $T_g^\CC K$ via the tangent map of the right translation by $g$. The restriction of $p_g^\CC$ to ${\mk n}$ (respectively, ${\mk n}^-$) is a $T$-module isomorphism onto $T_{gx_o}^{0,1}X$ (respectively, $T_{gx_o}^{1,0}X$). The holomorphic and the antiholomorphic tangent bundles on $X$ are $K$-homogeneous unitary vector bundles and we have
\begin{gather}\label{For n=T(0,1)}
T^{1,0}X=K\times_{T}{\mk n}^- \quad , \quad T^{0,1}X=K\times_{T}{\mk n} \;.
\end{gather}
In what follows we will make use of the Dolbeault $\ol{\partial}$-complex on $X$, hence we are specifically interested in antiholomorphic differential forms. The second identity in (\ref{For n=T(0,1)}) implies that the antiholomorphic cotangent bundle $\Omega^{0,1}X$ can be written as $K\times_{T}{\mk n}^*$, and more generally
$$
\Omega^{0,q}X=K\times_{T} (\Lambda^q{\mk n}^*) \;.
$$
The space of differentiable sections is
$$
C^{0,1}(X)=(C^{\infty}(K)\otimes {\mk n}^*)^{T} \;,
$$
where $T$ acts on $C^{\infty}(K)$ on the right.

\subsubsection{Sheaf cohomology and Lie algebra cohomology}\label{Sec Cohomologies}

A classical theorem of Lie states that the irreducible representations of solvable Lie groups (over $\CC$) are one-dimensional. Hence the only irreducible homogeneous holomorphic vector bundles on $X$ are line bundles. A homogeneous holomorphic line bundle on $X$ corresponds to a character of $B$. Since the character group of $B$ is identified with the weight lattice ${\mc P}$, and any holomorphic line bundle on a flag manifold is linearly equivalent to a homogeneous one, we have ${\rm Pic}(X)\cong {\mc P}$. In the introduction, we gave the notation ${\mc L}_\lambda= G\times_B \CC_\lambda$ for $\lambda\in{\mc P}$. The sheaf of local holomorphic sections of ${\mc L}_{-\lambda}$ was denoted by ${\mc O}(\lambda)$. The homogeneity of ${\mc L}_{-\lambda}$ yields a representation of $G$ on the cohomology space $H^q(X,{\mc O}(\lambda))$. All such representations are described by the Borel-Weil-Bott theorem formulated in the introduction.

The line bundle ${\mc L}_{-\lambda}\rw X$ can be written both as $G\times_{B}\CC_{-\lambda}$ and $K\times_{T}\CC_{-\lambda}$. We can form the Dolbeault $\ol{\partial}$-complex $C_{\ol{\partial}}(X,{\mc L}_{-\lambda})$ of antiholomorphic forms with values in ${\mc L}_{-\lambda}$. These forms can be lifted to forms on $K$, with the suitable invariance properties. First, we have the bundle
$$
\Omega^{0,q}(X,{\mc L}_{-\lambda})=K\times_{T} (\Lambda^q{\mk n}^* \otimes \CC_{-\lambda}) \;,
$$
whose space of differentiable sections is $C^{0,q}(X,{\mc L}_{-\lambda})=(C^{\infty}(K)\otimes \Lambda^q{\mk n}^* \otimes \CC_{-\lambda})^{T}$, which can also be written as ${\rm Hom}_{\mk h}(\Lambda^q{\mk n},C^{\infty}(K)\otimes \CC_{-\lambda})$. The total space of the $\ol{\partial}$-complex of antiholomorphic forms on $X$ with values in ${\mc L}_{-\lambda}$ is $(\Lambda{\mk n}^*\otimes C^{\infty}(K)\otimes \CC_{-\lambda})^{\mk h}$. The latter is also the total space of the Lie algebra cochain complex $C({\mk n},C^{\infty}(K)\otimes \CC_{-\lambda})^{\mk h}$ defined with respect to the right action of ${\mk n}$ on $C^{\infty}(K)$ and the trivial action on $\CC_{-\lambda}$. Moreover, the coboundary operators coincide and we have an isomorphism of complexes
$$
C_{\ol{\partial}}(X,{\mc L}_{-\lambda}) \cong C({\mk n},C^{\infty}(K)\otimes \CC_{-\lambda})^{\mk h} \;.
$$
Both complexes carry a left $K$-action and the isomorphism above is a $K$-module isomorphism. We can conclude that, for all $q$, there is a $K$-equivariant isomorphism of cohomology groups
\begin{gather}\label{For H^o,q(X)=H^q(n)}
H^{0,q}(X,{\mc L}_{-\lambda}) \cong H^q({\mk n},C^{\infty}(K)\otimes \CC_{-\lambda})^{\mk h} \;.
\end{gather}
Dolbeault's theorem gives \begin{gather}\label{For Dolbeault's thm}
H^{q}(X,{\mc O}(\lambda))\cong H^{0,q}(X,{\mc L}_{-\lambda}) \;.
\end{gather}
Combining (\ref{For H^o,q(X)=H^q(n)}) and (\ref{For Dolbeault's thm}) we deduce the following isomorphism of $K$-modules:
$$
H^{q}(X,{\mc O}(\lambda)) \cong H^q({\mk n},C^{\infty}(K)\otimes \CC_{-\lambda})^{\mk h} \;.
$$

We shall restrict our considerations to the space of representable functions ${\mc F}(K)$ instead of $C^{\infty}(K)$. This can be viewed as a restriction to differential forms with algebraic coefficients, if a suitable algebraic structure on $X$ is chosen. In fact, one has $H^q({\mk n},C^{\infty}(K)\otimes \CC_{-\lambda})^{\mk h} \cong H^q({\mk n},{\mc F}(K)\otimes \CC_{-\lambda})^{\mk h}$, which can be deduced, via the Peter-Weyl theorem, from the theorem of Kostant stated in the following section. Now, using the Peter-Weyl theorem we obtain
\begin{align*}
H^q({\mk n},{\mc F}(K)\otimes \CC_{-\lambda})^{\mk h} & = \oplus_{\mu\in{\mc P}^+} H^q({\mk n},V(\mu)^*\otimes V(\mu)\otimes \CC_{-\lambda})^{\mk h} \\
              & = \oplus_{\mu\in{\mc P}^+} V(\mu)^*\otimes \sco{H^q({\mk n},V(\mu))\otimes \CC_{-\lambda}}^{\mk h} \;.
\end{align*}
In particular, we arrive at the reciprocity law observed by Bott \cite{Bott}: for $\lambda\in{\mc P}$ and $\mu\in{\mc P}^+$ we have
\begin{gather}\label{For Bott Recip}
\dim {\rm Hom}_{G}(V(\mu)^*,H^{q}(X,{\mc O}(\lambda))) = \dim {\rm Hom}_{\mk h}(\CC_{\lambda},H^q({\mk n},V(\mu))) \;.
\end{gather}

\subsubsection{Kostant's theorem on Lie algebra cohomology}\label{Sec BWB and Kosta}

In the introduction we stated the Borel-Weil-Bott theorem, which determines the left hand side of (\ref{For Bott Recip}). In this section, we formulate a theorem of Kostant, which determines the right hand side. Both theorems are in fact more general - they treat general flag manifolds of $G$, i.e. coset spaces of the form $G/P$, where $P$ is a parabolic subgroup. We only present the statement for complete flag manifolds as it is sufficient for our purposes.

The cochain complex $C({\mk n},V(\mu))$ is endowed with a Hermitian form obtained as the product of the form $\{,\}$ on $\Lambda{\mk n}$, defined as the restriction of the Hermitian form on $\Lambda{\mk g}$ given in Section \ref{Par SSLieAlg}, and a $K$-invariant Hermitian form on $V(\mu)$. Let ${\rm L}=dd^*+d^*d$ be the Laplacian defined with respect to this Hermitian form. The main ingredient in Kostant's proof is a spectral decomposition for the Laplacian $L$ and, in particular, an explicit description of its kernel - the harmonic cocycles.

Let $\{ e^*_{-\alpha} \; \vert \; \alpha\subset\Delta^+ \}$ denote the basis of ${\mk n}^*$ dual to the basis $\{ e_{\alpha} \; \vert \; \alpha\subset\Delta^+ \}$ of ${\mk n}$. Thus $\{ e^*_{-\Phi}=\wedge_{\alpha\in\Phi}e^*_{-\alpha} \; \vert \; \Phi\subset\Delta^+ \}$ is a basis of $\Lambda{\mk n}^*$ dual to the basis $\{ e_{\Phi} \; \vert \; \Phi\subset\Delta^+ \}$ of $\Lambda{\mk n}$ (see Proposition \ref{Prop ol(a)=a* regular}).

\begin{theorem}\label{Theo Kostant's}{\rm(Kostant, \cite{Kostant})} Let $\mu\in{\mc P}^+$ and $q\in\NN$. Then the space of harmonics in $C^q({\mk n},V(\mu))$ is given by
$$
{\rm Harm}^q({\mk n},V(\mu)) = \oplus_{w\in{\mc W}(q)} C({\mk n},V(\mu))^{w\cdot\mu} \;,
$$
where the super-script $w\cdot\mu$ designates ${\mk h}$-weight space and ${\mc W}(q)$ denotes the set of elements in ${\mc W}$ with length $q$. For each $w\in {\mc W}$ the weight space $C({\mk n},V(\mu))^{w\cdot\mu}$ is one-dimensional, generated by the cocycle
$$
e^*_{-\Phi_{w^{-1}}} \otimes v^{w(\mu)} \;.
$$
Consequently, the cohomology group $H^q({\mk n},V(\mu))$ decomposes into one-dimensional ${\mk h}$-submodules as follows:
$$
H^q({\mk n},V(\mu)) = \oplus_{w\in{\mc W}(q)} H({\mk n},V(\mu))^{w\cdot\mu} \;.
$$
\end{theorem}

The above theorem can be used to obtain explicitly the harmonics for the complex $C({\mk n},{\mc F}(K)\otimes \CC_{-\lambda})^{\mk h}$, for regular $\lambda\in{\mc P}$. The relevant Hermitian form is the product of the form on $\Lambda{\mk n}$, as above, and the form on ${\mc F}(K)$ which is the restriction of the standard $K$-invariant form on $L^2(K)$. We have
\begin{gather*}
{\rm Harm}({\mk n},{\mc F}(K)\otimes \CC_{-\lambda})^{\mk h} = V(w_\lambda\cdot\lambda)^*\otimes ({\rm Harm}({\mk n},V(w_\lambda\cdot\lambda))^{\lambda}\otimes \CC_{-\lambda})^{\mk h} .
\end{gather*}
By the above theorem, ${\rm Harm}({\mk n},V(w_\lambda\cdot\lambda))^{\lambda} = e^*_{-\Phi_{w_\lambda}}\otimes v^{\lambda+\langle\Phi_{w_\lambda}\rangle}$. Here we have used the formula $w^{-1}(w\cdot\lambda)=\lambda+\langle\Phi_{w_\lambda}\rangle$ given in (\ref{For w^-1(w.lambda)}).

\section{Embeddings and cohomological components:\\ general results}\label{Sec Maps of Flags}

Let $\iota:{\mk g}\rw\wt{\mk g}$ be a monomorphism between two semisimple complex Lie algebras; we regard ${\mk g}$ as a subalgebra of $\wt{\mk g}$. Let $G$ and $\wt{G}$ the associated connected, simply connected complex Lie groups. Then $\iota$ induces a homomorphism $\phi: G\rw \wt{G}$ with a finite kernel.

We are interested in $\phi$-equivariant maps between flag manifolds of $G$ and $\wt{G}$. The complete flag manifold $X=G/B$ of $G$ can be viewed as a parameter space for all Borel subalgebras of ${\mk g}$. Defining a $G$-equivariant map from $X$ to the complete flag manifold $\wt{X}$ of $\wt{G}$ is equivalent to choosing a Borel subalgebra $\wt{\mk b}\subset\wt{\mk g}$ containing ${\mk b}$. Such a subalgebra always exists, but may not be unique in general. Since ${\mk b}={\mk g}\cap\wt{\mk b}$, the map $G/B\rw\wt{G}/\wt{B}$ is always an embedding. We now assume that a pair of Borel subalgebras ${\mk b}\subset\wt{\mk b}$ is fixed, and denote the resulting embedding of flag manifolds by
$$
\varphi:X\lw\wt{X} \;.
$$
Fix a compact real form ${\mk k}\subset{\mk g}$ and extend it to a compact real form $\wt{\mk k}\subset\wt{\mk g}$. Set ${\mk t}={\mk k}\cap{\mk b}$, $\wt{\mk t}=\wt{\mk k}\cap\wt{\mk b}$, ${\mk h}={\mk t}\oplus i{\mk t}$, $\wt{\mk h}=\wt{\mk t}\oplus i\wt{\mk t}$. Clearly ${\mk t}={\mk k}\cap\wt{\mk t}$ and ${\mk h}={\mk g}\cap\wt{\mk h}$. Denote by $K$, $\wt{K}$, $T$, $\wt{T}$, $H$, $\wt{H}$ the corresponding Lie groups. Note that the map $\varphi$ can be interpreted as a map between coset spaces of compact groups:
$$
\varphi:K/T \lw \wt{K}/\wt{T} \;.
$$
After the above choices have been made we can use the notions and notation given in the Introduction and Section \ref{Sec Preliminar}. The differential of $\varphi$ evaluated at $x_o=eT$ is a $T$-equivariant linear map, which we denote by $\varphi_o=(d\varphi)_{x_o}$. In view of (\ref{For n=T(0,1)}) we get
\begin{gather}\label{For varphi_o:n->wt(n)}
\varphi_o:{\mk n}^{\pm}\rw\wt{\mk n}^{\pm}
\end{gather}
Note that $\varphi_o$ coincides with the restriction of $\iota$ to ${\mk n}^{\pm}$, since ${\mk n}^{\pm}={\mk g}\cap\wt{\mk n}^{\pm}$. One should be aware that this coincidence is not necessarily true for partial flag manifolds.

\subsection{Pullback maps and cohomological components}\label{Sec Pullbacks and CohCom}

The homomorphism $\phi:G\lw \wt{G}$ and the embedding $\varphi:X\hookrightarrow\wt{X}$ defined in the previous section give rise to various pullback maps. Primarily, we are interested in the restriction of homogeneous holomorphic line bundles from $\wt{X}$ to $X$, and the resulting cohomological pullbacks. Since $\phi(B)\subset\wt{B}$, each character $\wt{\lambda}$ of $\wt{B}$ defines a character $\lambda=\wt{\lambda}\circ\phi$ of $B$, and for the associated line bundles we have
$$
\wt{\mc L}_{-\wt{\lambda}} \stackrel{\varphi^*}{\lo} {\mc L}_{-\lambda} \;.
$$
Recall that ${\rm Pic}(X)\cong{\mc P}$ and ${\rm Pic}(\wt{X})\cong\wt{\mc P}$ (see Section \ref{Sec Flags and Bundles}). Since $\iota:{\mk h}\rw\wt{\mk h}$ is an inclusion, its dual $\iota^*:\wt{\mk h}^*\rw{\mk h}^*$ is surjective. The map $\iota^*$ respects the weight lattices, hence the restriction of line bundles can be computed using the map
$$
\iota^*:\wt{\mc P}\lw{\mc P} \;.
$$
Note that the above map of lattices may not be surjective. In fact, we have ${\mc P}/\iota^*{\wt{\mc P}}\cong{\rm Ker}(\phi)$, the latter being a finite subgroup of $H$.

Let $\wt{\lambda}\in\wt{\mc P}$ and $\lambda=\iota^*(\wt{\lambda})\in{\mc P}$. There is a pullback map
$$
\pi^{\wt{\lambda}}: H(\wt{X},\wt{\mc O}(\wt{\lambda})) \lw H(X,{\mc O}(\lambda)) \;,
$$
which is $G$-equivariant. This map is the central object of our study. In particular, we are aiming to solve problems I,II,III stated in the introduction.

\begin{zab}\label{Zab H^0->H^0}
The three problems formulated in the introduction are easily solved for dominant $\wt{\lambda}$. Since ${\mk b}\subset\wt{\mk b}$, we have $\iota^*(\wt{\mc P}^+) \subset {\mc P}^+$. Let $\wt{\lambda}\in\wt{\mc P}^+$. Then

\noindent{\rm (I)} $\pi^{\wt{\lambda}}:H^0(\wt{X},\wt{\mc O}(\wt{\lambda})) \rw H^0(X,{\mc O}(\lambda))$ is nonzero, because the effective line bundles on homogeneous manifolds are base-point-free.

\noindent{\rm (II)} The corresponding cohomological component is $V(\lambda)\subset\wt{V}(\wt{\lambda})$, generated by a $\wt{\mk b}$-highest weight vector in $\wt{V}(\wt{\lambda})$.

\noindent{\rm (III)} Let ${\mc C}_0$ of denote the subset of ${\mc C}$ (see (\ref{For C in intro})) resulting from pullbacks on global sections. Then ${\mc C}_0=\{(\mu,\wt{\mu})\in{\mc P}^+\times\wt{\mc P}^+ : \mu=\iota^*(\wt{\mu})\}$.
\end{zab}

The above case is important but, as far as our problems are concerned, trivial. As we shall see, the problems become nontrivial for pullbacks in higher degrees of cohomology. To study $\pi^{\wt{\lambda}}$ we intend to translate it to a pullback in Lie algebra cohomology. First, Dolbeault's theorem allows us to translate the map $\pi^{\wt{\lambda}}$ to a pullback in Dolbeault cohomology, denoted by
$$
\varphi_{-\wt{\lambda}}^*: H^{0,q}(\wt{X},\wt{\mc L}_{-\wt{\lambda}}) \lw H^{0,q}(X,{\mc L}_{-\lambda}) \;.
$$
The nonvanishing of $\pi^{\wt{\lambda}}$ is equivalent to the nonvanishing of $\varphi_{-\wt{\lambda}}^*$. Now we can use the fact that the Dolbeault ${\ol{\partial}}$-complexes can be interpreted as Lie algebra cochain complexes (see Section \ref{Sec Cohomologies}).

The homomorphism $\phi:K\rw\wt{K}$ gives rise to a $K\times K$-equivariant pullback of representative functions, denoted by
$$
{\rm r}:{\mc F}(\wt{K}) \lw {\mc F}(K) \;.
$$
The tangent map $\varphi_o$ defined in (\ref{For varphi_o:n->wt(n)}), after dualization, gives rise to a map of Grassmann algebras
$$
\varphi_o^*: \Lambda\wt{\mk n}^* \lw \Lambda{\mk n}^* \;.
$$
Combining the maps ${\rm r}$ and $\varphi_o^*$ we get a pullback on Lie algebra cochain complexes
$$
\varphi_o^* \otimes {\rm r} : C(\wt{\mk n},{\mc F}(\wt{K})) \lw C({\mk n},{\mc F}(K)) \;.
$$
We denote the resulting map on cohomology by
$$
\varpi : H(\wt{\mk n},{\mc F}(\wt{K})) \lw H({\mk n},{\mc F}(K)) \;.
$$
The map $\varpi$ plays a central role, as it encompasses the information about all pullback maps $\pi^{\wt{\lambda}}$, for $\wt{\lambda}\in\wt{\mc P}$.

The pullback of antiholomorphic differential forms with values in ${\mc L}_{-\wt{\lambda}}$ and representative coefficients is translated to the following map between Lie algebra complexes
$$
\varphi_o^* \otimes {\rm r} \otimes 1 : C(\wt{\mk n},{\mc F}(\wt{K})\otimes\CC_{-\wt{\lambda}})^{\wt{\mk h}} \lw C({\mk n},{\mc F}(K)\otimes\CC_{-\lambda})^{\mk h} \;.
$$
The induced map on cohomology, denoted by $\varpi^{\wt{\lambda}}$, equals the restriction of $\varpi\otimes 1$ to the space of $\wt{\mk h}$-invariants:
$$
\varpi^{\wt{\lambda}} : H(\wt{\mk n},C^{\infty}(\wt{K})\otimes\CC_{-\wt{\lambda}})^{\wt{\mk h}} \lw H({\mk n},C^{\infty}(K)\otimes\CC_{-\lambda})^{\mk h}.
$$
In this way, we arrive at the following theorem.

\begin{theorem}\label{Prop Translate}
Let $\wt{\lambda}\in\wt{\mc P}$ and $\lambda=\iota^*(\wt{\lambda})\in{\mc P}$. Then the following diagram commutes:
\begin{gather*}
\begin{array}{ccccc}
H^q(\wt{X},\wt{\mc O}(\wt{\lambda})) & \cong & H^{0,q}(\wt{X},\wt{\mc L}_{-\wt{\lambda}}) & \cong & H^q(\wt{\mk n},C^{\infty}(\wt{K})\otimes\CC_{-\wt{\lambda}})^{\wt{\mk h}} \\
 &&&&\\
\pi^{\wt{\lambda}} \downarrow & & \varphi_{-\wt{\lambda}}^* \downarrow & & \varpi^{\wt{\lambda}} \downarrow \\
 &&&&\\
H^q(X,{\mc O}(\lambda)) & \cong & H^{0,q}(X,{\mc L}_{-\lambda}) & \cong & H^q({\mk n},C^{\infty}(K)\otimes\CC_{-\lambda})^{\mk h} \;.\\
\end{array}
\end{gather*}
\end{theorem}

\subsection{A criterion for nonvanishing of $\pi^{\wt{\lambda}}$}\label{Sec Criterion}

In this section we fix $\wt{\lambda}\in\wt{\mc P}$ and denote $\lambda=\iota^*(\wt{\lambda})\in{\mc P}$. We assume that both $\wt{\lambda}$ and $\lambda$ are regular weights (otherwise the associated pullback $\pi^{\wt{\lambda}}$ vanishes for trivial reasons) and we set $\wt{w}=\wt{w}_{\wt{\lambda}}$ and $w=w_{\lambda}$. The Borel-Weil-Bott theorem gives $H^{\wt{l}(\wt{\lambda})}(\wt{X},\wt{\mc O}(\wt{\lambda}))\cong \wt{V}(\wt{w}\cdot\wt{\lambda})^*$ and $H^{l(\lambda)}(X,{\mc O}(\lambda))\cong V(w\cdot\lambda)^*$. The map $\pi^{\wt{\lambda}}$ and its dual $(\pi^{\wt{\lambda}})^*$, written as maps of $G$-modules, are
$$
\pi^{\wt{\lambda}}:\wt{V}(\wt{w}\cdot\wt{\lambda})^* \lw V(w\cdot\lambda)^* \quad,\quad (\pi^{\wt{\lambda}})^*:V(w\cdot\lambda) \lw \wt{V}(\wt{w}\cdot\wt{\lambda}) \;.
$$
There are two obvious necessary conditions for nonvanishing of $\pi^{\wt{\lambda}}$: first, the equality $\wt{l}(\wt{\lambda})=l(\lambda)$; second, the existence of a $G$-submodule in $\wt{V}(\wt{w}\cdot\wt{\lambda})$ isomorphic to $V(w\cdot\lambda)$, i.e. ${\rm Hom}_{G}(V(w\cdot\lambda),\wt{V}(\wt{w}\cdot\wt{\lambda})) \ne 0$. These conditions are not sufficient in general; counter-examples may be found using the results of Section \ref{Sec Invariants}.

Theorem \ref{Prop Translate} allows us to translate the study of the map $\pi^{\wt{\lambda}}$ to the language of Lie algebra cohomology. Namely, the nonvanishing of $\pi^{\wt{\lambda}}$ is equivalent to the nonvanishing of $\varpi^{\wt{\lambda}}$. In this setting Kostant's results may be used to obtain a more explicit description of the pullback.

\begin{theorem}\label{Theo Necessity}
The pullback $\varpi^{\wt{\lambda}}$ (or, equivalently, $\pi^{\wt{\lambda}}$) is nonzero if and only if the following two conditions hold:

{\rm (i)} $\varphi_o^*(\wt{e}^*_{-\Phi_{\wt{w}}}) = a e^*_{-\Phi_{w}}$, for some nonzero complex number $a$.

{\rm (ii)} ${\rm Hom}_G(V(w\cdot\lambda),{\mk U}(\mk g)\wt{v}^{\wt{\lambda}+\langle\Phi_{\wt{w}}\rangle}) \ne 0$, i.e. the $G$-submodule of $\wt{V}(\wt{w}\cdot\wt{\lambda})$ generated by the extreme weight vector $\wt{v}^{\wt{\lambda}+\langle\Phi_{\wt{w}}\rangle}$ contains an irreducible component isomorphic to $V(w\cdot\lambda)$.
\end{theorem}

\begin{proof} It is clear that the map $\varpi^{\wt{\lambda}}$ is nonzero if and only if the restriction of the map $\varpi:H(\wt{\mk n}, {\mc F}(\wt{K})) \rw H({\mk n}, {\mc F}(K))$ to the $\wt{\mk h}$-submodule $H(\wt{\mk n}, {\mc F}(\wt{K}))^{\wt{\lambda}}$ is nonzero. Furthermore, we can restrict our attention to harmonic representatives, because the vanishing of the pullback on all harmonics is equivalent to total vanishing.

Kostant's theorem (see Theorem \ref{Theo Kostant's}) implies that, after decomposing the space ${\mc F}(\wt{K})$ according to the Peter-Weyl theorem, we get
$$
H(\wt{\mk n}, {\mc F}(\wt{K}))^{\wt{\lambda}} = H(\wt{\mk n},\wt{V}(\wt{w}\cdot\wt{\lambda})^*\otimes \wt{V}(\wt{w}\cdot\wt{\lambda}))^{\wt{\lambda}} \;.
$$
The harmonic representatives are pure tensors of the form $\wt{e}^*_{-\Phi_{\wt{w}}}\otimes \wt{f}\otimes \wt{v}^{\wt{\lambda}+\langle\Phi_{\wt{w}}\rangle}$, where $\wt{f}$ varies in $V(\wt{w}\cdot\wt{\lambda})^*$. Since $\iota^*(\wt{\lambda})=\lambda$ we have
$$
\varpi(H(\wt{\mk n}, {\mc F}(\wt{K}))^{\wt{\lambda}}) \subset H({\mk n},{\mc F}(K))^\lambda = H({\mk n}, V(w\cdot\lambda)^*\otimes V(w\cdot\lambda))^{\lambda} \;.
$$
Moreover, since the map of cochain complexes is also equivariant with respect to the right ${\mk h}$-action, we get
$$
\varphi_o^*(\wt{e}^*_{-\Phi_{\wt{w}}})\otimes {\rm r}(\wt{f}\otimes \wt{v}^{\wt{\lambda}+\langle\Phi_{\wt{w}}\rangle}) \in C({\mk n},{\rm r}(\wt{V}(\wt{w}\cdot\wt{\lambda})^*\otimes \wt{V}(\wt{w}\cdot\wt{\lambda})))^{\lambda} \;.
$$
After this preparation we are ready to prove the theorem.

Assume first that $\varpi^{\wt{\lambda}} \ne 0$. Then, by semisimplicity, $\wt{V}(\wt{w}\cdot\wt{\lambda})^*$ contains a $G$-submodule $V^*$ isomorphic to $V(w\cdot\lambda)^*$, such that ${\rm r}(V^*\otimes \wt{v}^{\wt{\lambda}+\langle\Phi_{\wt{w}}\rangle}) \ne 0$. Let $\wt{f}\in V^*$ be such that ${\rm r}(\wt{f}\otimes\wt{v}^{\wt{\lambda}+\langle\Phi_{\wt{w}}\rangle})\ne 0$. Then
$$
\varphi_o^*(\wt{e}^*_{-\Phi_{\wt{w}}})\otimes {\rm r}(\wt{f}\otimes \wt{v}^{\wt{\lambda}+\langle\Phi_{\wt{w}}\rangle}) \in C({\mk n},V(w\cdot\lambda)^*\otimes V(w\cdot\lambda))^{\lambda} = e^*_{-\Phi_w} \otimes V(w\cdot\lambda)^*\otimes V(w\cdot\lambda)^{\lambda+\langle\Phi_w\rangle} \;.
$$
Hence $\varphi^*_o(\wt{e}^*_{-\Phi_{\wt{w}}})$ is proportional to $e^*_{-\Phi_w}$, i.e. condition (i) is satisfied.

To show that condition (ii) holds as well, recall first that we have chosen maximal compact subgroups $K\subset G$ and $\wt{K}\subset\wt{G}$ such that ${\mk k}={\mk g}\cap\wt{\mk k}$. Weyl's unitary trick implies that there exists a unique (up to a scalar multiple) $\wt{K}$-invariant Hermitian form on $\wt{V}(\wt{w}\cdot\wt{\lambda})$. Fix one such form. The inclusion ${\mk k}\subset\wt{\mk k}$ implies that this form is also $K$-invariant. Consequently, the orthogonal complement $V^\perp$ of any $K$-submodule $V\subset\wt{V}(\wt{w}\cdot\wt{\lambda})$ is also a $K$-submodule. In particular, the $K$-isotypic decomposition of $\wt{V}(\wt{w}\cdot\wt{\lambda})$ is orthogonal, say $\wt{V}(\wt{w}\cdot\wt{\lambda})=U_1\oplus\dots\oplus U_m$. Of course, the isotypic decomposition of $\wt{V}(\wt{w}\cdot\wt{\lambda})$ as a $G$-module is the same. The extreme weight vector $\wt{v}^{\wt{\lambda}+\langle\Phi_{\wt{w}}\rangle}$ can be written uniquely as
$$
\wt{v}^{\wt{\lambda}+\langle\Phi_{\wt{w}}\rangle} = u_1 + \dots + u_m \quad,\quad u_j\in U_j \;.
$$
Clearly $u_j=0$ if and only if $\wt{v}^{\wt{\lambda}+\langle\Phi_{\wt{w}}\rangle}\in U_j^\perp$. The Hermitian form on $\wt{V}(\wt{w}\cdot\wt{\lambda})$ induces a vector space isomorphism $\wt{V}(\wt{w}\cdot\wt{\lambda})\cong\wt{V}(\wt{w}\cdot\wt{\lambda})^*$, as well as a Hermitian form on $\wt{V}(\wt{w}\cdot\wt{\lambda})^*$. The (orthogonal) $G$-isotypic decomposition of the dual space is $\wt{V}(\wt{w}\cdot\wt{\lambda})^*=U_1^*\oplus\dots\oplus U_m^*$. We have $u_j=0$ if and only if $\wt{f}(\wt{v}^{\wt{\lambda}+\langle\Phi_{\wt{w}}\rangle})=0$ for all $\wt{f}\in U_j^*$. On the other hand, the latter condition is equivalent to ${\rm r}(U_j^*\otimes \wt{v}^{\wt{\lambda}+\langle\Phi_{\wt{w}}\rangle})=0$.

Observe that
$$
{\mk U}(\mk g)\wt{v}^{\wt{\lambda}+\langle\Phi_{\wt{w}}\rangle}= {\mk U}(\mk g)u_1\oplus\dots\oplus {\mk U}(\mk g)u_m \;.
$$
Now, fix $U_j$ to be the isotypic component of $V(w\cdot\lambda)$. We have assumed that ${\rm r}(U_j^*\otimes \wt{v}^{\wt{\lambda}+\langle\Phi_{\wt{w}}\rangle})\ne 0$, which can now be seen to be equivalent to $u_j\ne 0$, which in turn holds if and only if ${\mk U}(\mk g)\wt{v}^{\wt{\lambda}+\langle\Phi_{\wt{w}}\rangle}$ contains a $G$-submodule isomorphic to $V(w\cdot\lambda)$. Thus condition (ii) holds.

To prove the other direction, suppose that conditions (i) and (ii) hold. Condition (i) implies that $\iota^*\langle\Phi_{\wt{w}}\rangle=\langle\Phi_w\rangle$. Condition (ii) implies that $\wt{V}^*$ contains a $G$-submodule $V^*$ isomorphic to $V(w\cdot\lambda)^*$ and there exists $\wt{f}\in V^*$ such that $\wt{f}(\wt{v}^{\wt{\lambda}+\langle\Phi_{\wt{w}}\rangle})\ne 0$. Then ${\rm r}(\wt{f}\otimes \wt{v}^{\wt{\lambda}+\langle\Phi_{\wt{w}}\rangle}) \in V(w\cdot\lambda)^*\otimes V(w\cdot\lambda)^{\lambda+\langle\Phi_w\rangle}$ is nonzero. It follows that the pullback of the harmonic cocycle $\wt{e}^*_{\Phi_{\wt{w}}}\otimes \wt{f}\otimes \wt{v}^{\wt{\lambda}+\langle\Phi_{\wt{w}}\rangle}\in C(\wt{\mk n},\wt{V}(\wt{w}\cdot\wt{\lambda})^*\otimes \wt{V}(\wt{w}\cdot\wt{\lambda}))^{\wt{\lambda}}$ is a nonzero cocycle in $C({\mk n},V(w\cdot\lambda)^*\otimes V(w\cdot\lambda))^{\lambda}$, and hence is harmonic. This implies that $\varpi^{\wt{\lambda}}\ne 0$.
\end{proof}

\begin{zab} The two conditions in Theorem \ref{Theo Necessity} are independent in general. Examples illustrating the independence are given in Remark \ref{Zab Invar i&ii indep}. In some cases, however, {\rm (i)} implies {\rm (ii)}, so that {\rm (i)} becomes a necessary and sufficient condition for nonvanishing. We will see that this holds for the classes of regular embeddings and diagonal embeddings. Another case when ${\rm (i)}$ implies ${\rm (ii)}$ is treated in Corollary \ref{Coro Triv modules} below. Such cases are particularly nice to work with, because condition {\rm (i)} is much easier to verify for a given inclusion of Lie algebras.
\end{zab}

\subsubsection{An alternative formulation of the criterion}

Theorem \ref{Theo Necessity}, specifically condition (ii) therein, can be reformulated in geometric terms. We first introduce some notation and a technical lemma, then state the reformulation as Theorem \ref{Theo Necessity}$'$.

For $\wt{w}\in\wt{\mc W}$, let $X_{\wt{w}}^o=G\wt{w}^{-1}\wt{B}/\wt{B}$ denote the $G$-orbit through $\wt{w}^{-1}\wt{B}$ in $\wt{X}$, and let $X_{\wt{w}}=\ol{X_{\wt{w}}^o}$ denote the closure in $\wt{X}$. For $\wt{\nu}\in\wt{\mc P}^+$ let ${\mc O}_{\wt{w}}(\wt{\nu})$ denote the restriction of $\wt{\mc O}(\wt{\nu})$ to $X_{\wt{w}}$, and let
\begin{gather}\label{For pi_w^lamb}
\pi_{\wt{w}}^{\wt{\nu}} : H(\wt{X},\wt{\mc O}(\wt{\nu})) \lw H(X_{\wt{w}},{\mc O}_{\wt{w}}(\wt{\nu}))
\end{gather}
denote the corresponding pullback on cohomology.

For $\wt{\mu}\in\wt{\mc P}^+$ there is a $\wt{G}$-equivariant map
\begin{gather*}
\begin{array}{cccc}
\wt{\psi}_{\wt{\mu}} : & \wt{X} & \lw & \PP(\wt{V}(\wt{\mu})) \\
 & \wt{g}\wt{B} & \lo & [\wt{g}\wt{v}^{\wt{\mu}}]
\end{array}
\end{gather*}
Note that
$$
\wt{\psi}_{\wt{\mu}}(\wt{w}^{-1}\wt{B}) = [\wt{v}^{\wt{w}^{-1}(\wt{\mu})}] \;.
$$

\begin{lemma}\label{Lem Im(pi_w)=Gv^w(mu)}
Let $\wt{w}\in\wt{\mc W}$ and $\wt{\mu}\in\wt{\mc P}^+$. Then there is a $G$-module isomorphism
$$
Im(\pi_{\wt{w}}^{\wt{\mu}}) \cong ({\mk U}(\mk g)\wt{v}^{\wt{w}^{-1}(\wt{\mu})})^* \;,
$$
where $\wt{v}^{\wt{w}^{-1}(\wt{\mu})}$ is a non-zero vector in $\wt{V}(\wt{\mu})$ of the extreme weight $\wt{w}^{-1}(\wt{\mu})$ and ${\mk U}(\mk g)\wt{v}^{\wt{w}^{-1}(\wt{\mu})}$ is the $G$-submodule of $\wt{V}(\wt{\mu})$ generated by this vector.
\end{lemma}

\begin{proof} The lemma follows immediately from the definitions.
\end{proof}

Lemma \ref{Lem Im(pi_w)=Gv^w(mu)} provides an equivalent substitute for condition (ii) of Theorem \ref{Theo Necessity} and allows us to formulate the following version of the criterion for nonvanishing of $\pi^{\wt{\lambda}}$.\\

\noindent{\bf Theorem \ref{Theo Necessity}}$'$. {\it
Let $\wt{\lambda},\lambda,\wt{w},w$ be as in the beginning of section \ref{Sec Criterion}. The pullback $\varpi^{\wt{\lambda}}$ (or, equivalently, $\pi^{\wt{\lambda}}$) is nonzero if and only if the following two conditions hold:

{\rm (i)} $\varphi_o^*(\wt{e}^*_{-\Phi_{\wt{w}}}) = a e^*_{-\Phi_{w}}$, for some nonzero complex number $a$.

{\rm (ii)}$'$ ${\rm Hom}_G(V(w\cdot\lambda)^*,Im(\pi_{\wt{w}}^{\wt{w}\cdot\wt{\lambda}})) \ne 0$, i.e. the space $H^0(\wt{X},\wt{\mc O}(\wt{w}\cdot\wt{\lambda}))$ contains a $G$-submodule isomorphic to $V(w\cdot\lambda)^*$ whose elements do not vanish identically on $X_{\wt{w}}$.
}\\

\subsubsection{Corollaries}

Below we formulate some corollaries from Theorems \ref{Theo Necessity} and \ref{Theo Necessity}$'$.

\begin{coro} Let $\mu\in{\mc P}^+$ and $\wt{\mu}\in\wt{\mc P}^+$. Suppose that $\wt{V}(\wt{\mu})$ contains $V(\mu)$ as a cohomological component, i.e. the projection $V(\wt{\mu})^*\rw V(\mu)^*$ can be realized as a cohomological pullback $\pi^{\wt{\sigma}\cdot\wt{\mu}}: H(\wt{X},{\mc O}(\wt{\sigma}\cdot\wt{\mu})) \rw H(X,{\mc O}(\sigma\cdot\mu))$, with $\wt{\sigma}\in\wt{\mc W}$ and $\sigma\in{\mc W}$. Then $V(\mu)\subset {\mk U}(\mk g)\wt{v}^{\wt{\sigma}(\wt{\mu})}$.
\end{coro}

\begin{proof} The corollary follows directly from the proof of Theorem \ref{Theo Necessity}.
\end{proof}

\begin{coro}\label{Coro Triv modules}
In the notation of Theorem \ref{Theo Necessity}, suppose $\wt{\lambda}=\wt{w}^{-1}\cdot 0$. Then condition {\rm (i)} is necessary and sufficient for the nonvanishing of $\varpi^{\wt{\lambda}}$ (or equivalently of $\pi^{\wt{\lambda}}$). Furthermore, if condition ${\rm (i)}$ holds, then $\pi^{\wt{\lambda}}$ is a nonzero morphism from the one dimensional $\wt{G}$-module to the one dimensional $G$-module.
\end{coro}

\begin{proof}
The corollary follows immediately from Theorem \ref{Theo Necessity}, in view of the fact that $H(\wt{X}, \wt{\mc O}(\wt{\lambda})) \cong \CC$ if and only if $\wt{\lambda}=\wt{\sigma}\cdot 0$ for some $\wt{\sigma}\in\wt{\mc W}$. A simpler, independent proof is sketched below.

The cohomological morphisms between trivial modules are all captured in the pullback of Lie algebra cohomology with trivial coefficients:
\begin{gather*}
\begin{array}{cccc}
\varpi_o :& H(\wt{\mk n}) & \lw & H(\mk n) \\
        &  [\eta] & \lo & [\varphi_o^*(\eta)] \;,
\end{array}
\end{gather*}
where, as usual, $H(\mk n)$ stands for $H({\mk n},\CC)$. We have $H({\mk n})=H({\mk n},{\mc F}(K))^{\mk g}$ with respect to the left ${\mk g}$-action. Similarly $H(\wt{\mk n})=H(\wt{\mk n},{\mc F}(\wt{K}))^{\wt{\mk g}}$. The weight space decompositions of the two cohomology spaces above, as right modules over $\wt{\mk h}$ and ${\mk h}$ respectively, are:
\begin{gather}\label{For LieCohTriv}
\begin{array}{ll}
H(\wt{\mk n}) = \oplus_{\wt{\sigma}\in\wt{\mc W}} H(\wt{\mk n})^{\wt{\sigma}\cdot 0} \quad ,& \quad  H(\wt{\mk n})^{\wt{\sigma}\cdot 0} = \CC[\wt{e}^*_{-\Phi_{\wt{\sigma}^{-1}}}]\; ,\\
& \\
H(\mk n) = \oplus_{\sigma\in{\mc W}} H(\mk n)^{\sigma\cdot 0} \quad ,& \quad H(\mk n)^{\sigma\cdot 0} = \CC[e^*_{-\Phi_{\sigma^{-1}}}]\; .
\end{array}
\end{gather}
The map $\varpi_o$ is ${\mk h}$-equivariant and its nonvanishing on a given $\wt{\mk h}$-weight space corresponds to a cohomological morphism between trivial modules.
\end{proof}

In the next corollary we present a sufficient condition for nonvanishing of $\pi^{\wt{\lambda}}$, deduced from Theorem \ref{Theo Necessity} and a result of Montagard, Pasquier and Ressayre, \cite{Ressayre-et-al}. We need some preliminaries.

The homology of the flag manifold $X$ with integral coefficients is classically computed using the Schubert cell decomposition. The Schubert cells are the $B$-orbits in $X$. They are parametrized by ${\mc W}$: the cell decomposition is $X=\sqcup_{w\in{\mc W}} BwB/B$; the dimension of the orbit $BwB/B$ is $l(w)$. Let $\Omega_w=\ol{BwB/B}$ be closure in $X$. The Schubert cycles $\Omega_w$ generate $H_{\cdot}(X,\ZZ)$ and form a basis for $H_{\cdot}(X,\CC)$. Similarly, the Schubert cycles $\wt{\Omega}_{\wt{w}}=\ol{\wt{B}\wt{w}\wt{B}/\wt{B}}$ for $\wt{w}\in\wt{\mc W}$ generate $H_{\cdot}(\wt{X},\ZZ)$.

Any algebraic subvariety of $Y\subset \wt{X}$ is a cocycle, and its cohomology class $[Y]$ may be expressed as a non-negative linear combination of $\wt{\Omega}_{\wt{w}}$:
$$
[Y]=\sum\limits_{\wt{w}\in\wt{\mc W}} c_{\wt{w}}(Y) \wt{\Omega}_{\wt{w}} \;, \quad c_{\wt{w}}(Y)\in\ZZ_{\geq 0} \;.
$$
A subvariety $Y\subset\wt{X}$ is called multiplicity free, if $c_{\wt{w}}(Y)\in\{0,1\}$ for all $\wt{w}\in\wt{\mc W}$.

\begin{coro}\label{Coro Multi-free}
Let $\wt{\lambda}\in{\mc P}$ and put $\lambda=\iota^*(\wt{\lambda})$. Suppose that both $\wt{\lambda}$ and $\lambda$ are regular and denote $\wt{w}=\wt{w}_{\wt{\lambda}}$ and $w=w_\lambda$. Suppose further that the following two conditions hold:

{\rm (i)} $\varphi_o^*(\wt{e}^*_{-\Phi_{\wt{w}}}) = a e^*_{-\Phi_w}$ for some $a\in\CC^\times$.

{\rm (ii)} The closure $X_{\wt{w}}=\ol{G\wt{w}^{-1}\wt{B}/\wt{B}}$ of the $G$-orbit through $\wt{w}^{-1}$ in $\wt{X}$ is multiplicity free.

Then the pullback $\pi^{\wt{\lambda}}:H(\wt{X},\wt{\mc O}(\wt{\lambda})) \lw H(X,{\mc O}(\lambda))$ is non-zero.
\end{coro}

\begin{proof} Condition (i) of the present corollary is the same as condition (i) of Theorem \ref{Theo Necessity}$'$. To prove the corollary it is sufficient to show that condition (ii) of the present corollary implies condition (ii)$'$ of Theorem \ref{Theo Necessity}$'$.

Denote $\wt{\mu}=\wt{w}\cdot\wt{\lambda}$ and $\mu=w\cdot\lambda$, so that $H(\wt{X},\wt{\mc O}(\wt{\lambda}))\cong\wt{V}(\wt{\mu})^*$ and $H(X,{\mc O}(\lambda))\cong V(\mu)^*$ by the Borel-Weil-Bott theorem. We invoke Theorem 1 from \cite{Ressayre-et-al}, which states that, if $X_{\wt{w}}$ is multiplicity free in $\wt{X}$, then we have ${\rm Hom}_G(V(\mu),\wt{V}(\wt{\mu}))\ne 0$. The proof actually contains a stronger result. Namely, the assertion that the pullback (see (\ref{For pi_w^lamb}) for the notation)
\begin{gather*}
\pi_{\wt{w}}^{\wt{\mu}} : H(\wt{X},\wt{\mc O}(\wt{\mu})) \lw H(X_{\wt{w}},{\mc O}_{\wt{w}}(\wt{\mu}))
\end{gather*}
is surjective and its image contains a $G$-submodule isomorphic to $V(\mu)^*$. In particular, condition (ii)$'$ of Theorem \ref{Theo Necessity}$'$ holds.
\end{proof}

\begin{zab}\label{Zab Mult-Free}
There are classes of inclusions $\iota:{\mk g}\hookrightarrow\wt{\mk g}$ for which the closure of any $G$-orbit in $\wt{G}/\wt{B}$ is multiplicity free. For instance, when the image of $G$ in $\wt{G}$ is a spherical subgroup of minimal rank, cf \cite{Ressayre-SphericMinRank}. In such a case, condition (i) of Theorem \ref{Coro Multi-free} becomes a necessary and sufficient condition for nonvanishing of $\pi^{\wt{\lambda}}$. We shall discuss one such case independently, namely, the diagonal embedding $\phi:G\hookrightarrow G\times G$, in Section \ref{Sec Diagonal}.
\end{zab}

\subsection{Properties of cohomological components}\label{Sec PropCohCom}

Recall that ${\rm E}$ denotes the real span of $\Delta$ in ${\mk h}^*$. Similarly, let $\wt{\rm E}$ denote the real span of $\wt{\Delta}$ in $\wt{\mk h}^*$. Then ${\mc P}\times\wt{\mc P}\subset{\rm E}\oplus\wt{\rm E}$. Consider the monoid ${\mc D}={\mc P}^+\times\wt{\mc P}^+$ whose elements are pairs of dominant weights for $G$ and $\wt{G}$, respectively. Problem III stated in the introduction is concerned with the characterization of the set ${\mc C}={\mc C}(\varphi)$ of elements $(\mu,\wt{\mu})\in{\mc D}$ such that there exists a cohomological component $V(\mu)\subset\wt{V}(\wt{\mu})$ associated with the embedding $\varphi:X\hookrightarrow\wt{X}$. This goal is not achieved here, however, we prove some important facts about ${\mc C}$ showing that this set has structure. We start with the following technical lemma.

\begin{lemma}\label{Lem D(w,wt(w))}
Let ${\mc D}={\mc P}^+\times\wt{\mc P}^+$. Suppose that $\wt{w}\in\wt{\mc W}$ and $w\in{\mc W}$ satisfy condition {\rm (i)} of Theorem \ref{Theo Necessity}, i.e. there exists $a\in\CC^\times$ such that\\

\noindent {\rm (i)} \indent $\varphi_o^*(\wt{e}^*_{-\Phi_{\wt{w}}}) = a e^*_{-\Phi_w}$.\\

\noindent The set
$$
{\mc D}_{w,\wt{w}} = \{ (\mu,\wt{\mu})\in{\mc D} : \mu=w\cdot(\iota^*(\wt{w}^{-1}\cdot\wt{\mu})) \}
$$
is a finitely generated submonoid of ${\mc D}$.
\end{lemma}

\begin{proof}
Condition (i) implies that $\iota^*\langle\Phi_{\wt{w}}\rangle=\langle\Phi_w\rangle$, which in turn implies that for $\wt{\mu}\in\wt{\mc P}^+$,
\begin{gather}\label{For (i)->Demaz=Lin}
w\cdot(\iota^*(\wt{w}^{-1}\cdot\wt{\mu})) = w(\iota^*(\wt{w}^{-1}(\wt{\mu}))) \;.
\end{gather}
The right hand side depends linearly on $\wt{\mu}$. Thus ${\mc D}_{w,\wt{w}}$ is the intersection of ${\mc D}$ with a linear subspace of ${\rm E}\oplus\wt{\rm E}$. This implies that ${\mc D}_{w,\wt{w}}$ is a submonoid of ${\mc D}$. Since both ${\mc D}$ and ${\mc P}\times\wt{\mc P}$ are finitely generated, so is ${\mc D}_{w,\wt{w}}$.
\end{proof}

Theorem \ref{Theo Necessity} implies that
$$
{\mc C} = \bigcup_{w,\wt{w}} {\mc C}\cap{\mc D}_{w,\wt{w}} \;,
$$
where the union is over the pairs $w,\wt{w}$ satisfying (i). Denote ${\mc C}_{w,\wt{w}} = {\mc C}\cap{\mc D}_{w,\wt{w}}$. The main result of this section is the following theorem.

\begin{theorem}\label{Theo C_w,wt(w) monoid}
Suppose that $\wt{w}$ and $w\in{\mc W}$ are as in Lemma \ref{Lem D(w,wt(w))}. Then ${\mc C}_{w,\wt{w}}$ is a submonoid of ${\mc D}_{w,\wt{w}}$. Furthermore, there exists a positive integer $k$, depending only on the embedding $\varphi$ and the elements $w,\wt{w}$, such that for any $(\mu,\wt{\mu})\in{\mc D}_{w,\wt{w}}$ we have $(k\mu,k\wt{\mu})\in{\mc C}_{w,\wt{w}}$.
\end{theorem}

\begin{proof} First, from Corollary \ref{Coro Triv modules} we obtain $(0,0)\in{\mc C}_{w,\wt{w}}$. The theorem follows directly from Lemma \ref{Lem Asympto} and Lemma \ref{Lem Semigroup} which we state and prove below.
\end{proof}

\begin{zab}\label{Zab Semigroups}
A question naturally arising after Theorem \ref{Theo C_w,wt(w) monoid} is: what is the minimal $k$ and how can it be computed? Various situations may occur. For instance, as we shall see, for regular and diagonal embeddings we always have ${\mc C}_{w,\wt{w}}={\mc D}_{w,\wt{w}}$, so that $k=1$. On the other hand, the results of Section \ref{Sec Invariants} allow us to construct examples for which the complement ${\mc D}_{w,\wt{w}}\setminus{\mc C}_{w,\wt{w}}$ is non-empty, finite in some cases and infinite in others; see Remark \ref{Zab Invar&Monoid}. Corollary \ref{Coro Multi-free} implies that, when the orbit-closure $X_{\wt{w}}$ is multiplicity free in $\wt{X}$, we have $k=1$. It would be interesting to know whether there is a relation between the coefficients $c_{\wt{\sigma}}(X_{\wt{w}})$ and $k$.
\end{zab}

\begin{lemma}\label{Lem Asympto}
Suppose that $\wt{w}\in\wt{\mc W}$ and $w\in{\mc W}$ are as in Lemma \ref{Lem D(w,wt(w))}. Then there exists a positive integer $k$ such that for any $(\mu,\wt{\mu})\in{\mc D}_{w,\wt{w}}$ the pullback
$$
\pi^{\wt{w}^{-1}\cdot k\wt{\mu}} : \wt{V}(k\wt{\mu})^* \lw V(k\mu)^* \;.
$$
is non-zero. In particular, $(k\mu,k\wt{\mu})\in {\mc C}_{w,\wt{w}}$.
\end{lemma}

\begin{proof} Let $(\mu,\wt{\mu})\in{\mc D}_{w,\wt{w}}$. For $k\in\ZZ_{>0}$ put $\wt{\lambda}_k=\wt{w}^{-1}\cdot k\wt{\mu}$, $\lambda_k=\iota^*(\wt{\lambda}_k)$ and $\mu_k=w\cdot\lambda_k$, so that $\mu=\mu_1$. Using (\ref{For (i)->Demaz=Lin}) we get
$$
\mu_k = w\cdot(\iota^*(\wt{w}^{-1}\cdot k\wt{\mu})) = w(\iota^*(\wt{w}^{-1}(k\wt{\mu}))) = k\mu \;.
$$
Recall that $X_{\wt{w}}^\circ = G\wt{w}\wt{B}/\wt{B}\subset \wt{X}$ denotes the $G$-orbit in $\wt{X}$ through the point corresponding to the element $\wt{w}^{-1}$.

Condition (i) of Theorem \ref{Theo Necessity}$'$ holds by hypothesis. Condition (ii)$'$ is satisfied for $\pi^{\wt{w}^{-1}\cdot k\wt{\mu}}$ if and only if there exists a $G$-submodule $V(k\mu)^*\subset H(\wt{X},\wt{\mc O}(k\wt{\mu}))$ whose elements do not vanish identically on $X_{\wt{w}}^\circ$. The existence of $k$, for which such a $G$-submodule exists, is asserted by Lemma 2 of \cite{Ressayre-et-al}. Now Theorem \ref{Theo Necessity}$'$ implies that $\pi^{\wt{w}^{-1}\cdot k\wt{\mu}}\ne 0$.

The fact that $k$ can be chosen uniformly for all $(\mu,\wt{\mu})\in{\mc D}_{w,\wt{w}}$ follows from the fact that ${\mc D}_{w,\wt{w}}$ is finitely generated.
\end{proof}

\begin{lemma}\label{Lem Semigroup}
Suppose that $\wt{w}$ and $w\in{\mc W}$ are as in Lemma \ref{Lem D(w,wt(w))}. Suppose further that $(\mu_1,\wt{\mu}_1)$ and $(\mu_2,\wt{\mu}_2)$ belong to ${\mc C}_{w,\wt{w}}$, so that both pullbacks
$$
\pi^{\wt{w}^{-1}\cdot\wt{\mu}_1} : \wt{V}(\wt{\mu}_1)^* \lw V(\mu_1)^* \quad , \quad \pi^{\wt{w}^{-1}\cdot\wt{\mu}_2} : \wt{V}(\wt{\mu}_2)^* \lw V(\mu_2)^*
$$
are non-zero. Denote $\wt{\mu}=\wt{\mu}_1+\wt{\mu}_2$ and $\mu=\mu_1+\mu_2$. Then $(\mu,\wt{\mu})\in{\mc C}_{w,\wt{w}}$ and the pullback
$$
\pi^{\wt{w}^{-1}\cdot\wt{\mu}} : \wt{V}(\wt{\mu})^* \lw V(\mu)^*
$$
is non-zero. In other words, ${\mc C}_{w,\wt{w}}$ is a subsemigroup of ${\mc D}_{w,\wt{w}}$.
\end{lemma}

\begin{proof}
Denote $\wt{\lambda}_j=\wt{w}^{-1}\cdot\wt{\mu}_j$ and $\lambda_j=\iota^*(\wt{\lambda}_j)$ for $j=1,2$. Also, put $\wt{\lambda}=\wt{w}^{-1}\cdot\wt{\mu}$ and $\lambda=\iota^*(\wt{\lambda})$. Lemma \ref{Lem D(w,wt(w))} implies that $(\mu,\wt{\mu}) \in {\mc D}_{w,\wt{w}}$ and $\lambda=w^{-1}\cdot\mu$. Thus the pullback $\pi^{\wt{\lambda}}$ is indeed a map from $\wt{V}(\wt{\mu})^*$ to $V(\mu)^*$.

To check that $\pi^{\wt{\lambda}}\ne 0$ we shall apply Theorem \ref{Theo Necessity}$'$. By assumption, condition (i) is satisfied by $\pi^{\wt{\lambda}}$. It remains to verify condition (ii)$'$. It is convenient to relabel condition (ii)$'$ as applied to the pullbacks $\pi^{\wt{\lambda}_1}$, $\pi^{\wt{\lambda}_2}$ and $\pi^{\wt{\lambda}}$, as follows:\\

\noindent (ii)$'_1$ \indent ${\rm Hom}_G(V(\mu_1)^*,Im(\pi_{\wt{w}}^{\wt{\mu}_1})) \ne 0$.\\

\noindent (ii)$'_2$ \indent ${\rm Hom}_G(V(\mu_2)^*,Im(\pi_{\wt{w}}^{\wt{\mu}_2})) \ne 0$.\\

\noindent (ii)$'$ \indent ${\rm Hom}_G(V(\mu)^*,Im(\pi_{\wt{w}}^{\wt{\mu}})) \ne 0$.\\

Our hypothesis for nonvanishing of $\pi^{\wt{\lambda}_1}$ and $\pi^{\wt{\lambda}_2}$ implies, via Theorem \ref{Theo Necessity}$'$, that (ii)$'_1$ and (ii)$'_2$ hold. We need to verify (ii)$'$. Thus, we have $G$-submodules $V_1^*\subset H^0(\wt{X},\wt{O}(\wt{\mu}_1))$ and $V_2^*\subset H^0(\wt{X},\wt{O}(\wt{\mu}_2))$, isomorphic respectively to $V(\mu_1)^*$ and $V(\mu_2)^*$, whose elements do not vanish identically on $X_{\wt{w}}$. Let $\delta:\wt{X}\hookrightarrow\wt{X}\times\wt{X}$ denote the diagonal embedding. It follows that $H^0(\wt{X}\times\wt{X},\wt{\mc O}(\wt{\mu}_1)\boxtimes\wt{\mc O}(\wt{\mu}_2))$ contains the $G\times G$-submodule $V_1^*\otimes V_2^*$ whose elements do not vanish identically on $\delta(X_{\wt{w}})\subset\wt{X}\times\wt{X}$. Let $V_3^*\subset V_1^*\otimes V_2^*$ be the highest $G$-irreducible component; we have $V_3^*\cong V(\mu)^*$. Then $V_3^*$ does not vanish identically on $\delta(X_{\wt{w}})$. Notice that the restriction of $H^0(\wt{X}\times\wt{X},\wt{\mc O}(\wt{\mu}_1)\boxtimes\wt{\mc O}(\wt{\mu}_2))$ to the diagonal $\delta(\wt{X})$ is a surjection onto $H^0(\wt{X},\wt{\mc O}(\wt{\mu}))$. Thus any irreducible $G$-submodule of $H^0(\wt{X}\times\wt{X},\wt{\mc O}(\wt{\mu}_1)\boxtimes\wt{\mc O}(\wt{\mu}_2))$ which has a non-zero pullback to $\delta(X_{\wt{w}})$ is isomorphic to an irreducible $G$-submodule of $H^0(\wt{X},\wt{\mc O}(\wt{\mu}))$ which has a nonzero pullback to $X_{\wt{w}}$. In particular, there exists $V^*\subset H^0(\wt{X},\wt{\mc O}(\wt{\mu}))$ such that $V^*\cong V_3^*\cong V(\mu)^*$ and such that $V^*$ does not vanish identically on $X_{\wt{w}}$. Hence (ii)$'$ holds and $\pi^{\wt{\lambda}}$ is non-zero. This completes the proof.
\end{proof}

\section{Special cases and examples}

\subsection{Regular embeddings}\label{Sec Regular}

In this section we consider the case when $\iota:{\mk g}\subset\wt{\mk g}$ is a regular subalgebra, i.e. we suppose that a Cartan subalgebra $\wt{\mk h}\subset\wt{\mk g}$ is given such that ${\mk g}$ is an $\wt{\mk h}$-submodule of $\wt{\mk g}$. Then ${\mk h}={\mk g}\cap\wt{\mk h}$ is a Cartan subalgebra of ${\mk g}$ and the corresponding root system $\Delta=\Delta({\mk g},{\mk h})$ is included as a subset in the root system $\wt{\Delta}=\Delta(\wt{\mk g},\wt{\mk h})$, namely $\Delta=\{\wt{\alpha}\in\wt{\Delta}:\mk{g}\supset\wt{\mk g}^{\wt{\alpha}}\}\subset\wt{\Delta}$. The Weyl group ${\mc W}$ is included in $\wt{\mc W}$ as the subgroup generated by the reflections along the roots in $\Delta$. Notice that in this setting, the homomorphism $\phi:G\rw\wt{G}$ of simply connected Lie groups is a monomorphism and we can think of $G$ as a subgroup of $\wt{G}$.

Assume that ${\mk b}$ and $\wt{\mk b}$ are Borel subalgebras of ${\mk g}$ and $\wt{\mk g}$ respectively, such that ${\mk b}={\mk g}\cap\wt{\mk b}$. Let
$$
\varphi : X \hookrightarrow \wt{X}
$$
be the associated embedding of flag manifolds.

The above choice of Borel subalgebras defines partitions $\Delta=\Delta^+\sqcup\Delta^-$ and $\wt{\Delta}=\wt{\Delta}^+\sqcup\wt{\Delta}^-$, which satisfy $\Delta^{\pm}=\Delta\cap\wt{\Delta}^{\pm}$. Let $\Pi\subset\Delta^+$ and $\wt{\Pi}\subset\wt{\Delta}^+$ be the sets of simple roots. Note that we have $\wt{\Pi}\cap\Delta\subset\Pi$, but not necessarily $\Pi\subset\wt{\Pi}$. As a consequence, the length of Weyl group elements may change under the inclusion ${\mc W}\subset\wt{\mc W}$.

The restriction of weights $\iota^*:\wt{\mk h}^*\lw \mk{h}^*$ is a morphism of ${\mc W}$-modules. More precisely, consider the subspaces of $\wt{\mk h}^*$ defined as $\mk{h}^*_0={\rm Ker}(\iota^*)$ and $\mk{h}^*_1=span_\CC\{\Delta\}$. Then $\wt{\mk h}^*=\mk{h}^*_0\oplus\mk{h}^*_1$ is direct sum of ${\mc W}$-modules orthogonal with respect to $\wt{\kappa}$. The restriction of $\iota^*$ to $\mk{h}^*_1$ is an isomorphism onto $\mk{h}^*$ preserving the root system $\Delta$, the ${\mc W}$ action, and mapping the integral weights $\wt{\mc P}\cap\mk{h}^*_1$ into the weight lattice ${\mc P}$. The action of ${\mc W}$ on ${\mk h}^*_0$ is trivial.

Let $h_{\alpha},e_{\alpha}$, for $\alpha\in\Delta$ be coroots and root vectors for ${\mk g}$, as in Section \ref{Par SSLieAlg}. Analogously, let $\wt{h}_{\wt{\alpha}},\wt{e}_{\wt{\alpha}}$, for $\wt{\alpha}=\alpha\in \wt{\Delta}$ be coroots and root vectors for $\wt{\mk g}$. In a case when $\wt{\alpha}\in\Delta$ we have a proportion of the root vectors $\wt{e}_{\wt{\alpha}}=ae_{\alpha}$ with $a\in\CC\setminus\{0\}$. Dualizing, we get the following lemma.

\begin{lemma}\label{Lem EdinRootRestric}
Let $\wt{\alpha}\in\wt{\Delta}^+$. The following are equivalent:

{\rm (i)} The root $\wt{\alpha}$ belongs to $\Delta$.

{\rm (ii)} $\varphi_o^*(\wt{e}^*_{-\wt{\alpha}}) \ne 0$.

{\rm (iii)} $\varphi_o^*(\wt{e}^*_{-\wt{\alpha}})=ae^*_{-\alpha}$ for some $\alpha\in \Delta$ and $a\in\CC\setminus\{0\}$.
\end{lemma}

We can now derive another lemma.

\begin{lemma}\label{Lem HarmRestrictionForRegular}
Let $\wt{w}\in\wt{\mc W}$. Then the following are equivalent:

{\rm (i)} If $\wt{w}=\wt{s}_{\wt{\alpha}_1}\dots \wt{s}_{\wt{\alpha}_q}$ is any reduced expression for $\wt{w}$ in the generators $\wt{s}_{\wt{\alpha}},\wt{\alpha}\in\wt{\Pi}$, then $\wt{\alpha}_1,...,\wt{\alpha}_q$ are simple roots in $\Delta^+$.

{\rm (ii)} $\wt{w}\in{\mc W}$ and $l(\wt{w})=\wt{l}(\wt{w})$.

{\rm (iii)} $\Phi_{\wt{w}}\subset\Delta$.

{\rm (iv)} $\varphi_o^*(\wt{e}^*_{-\Phi_{\wt{w}}})\ne 0$.

{\rm (v)} $\varphi_o^*(\wt{e}^*_{-\Phi_{\wt{w}}}) = a e^*_{-\Phi_w}$ for some $w\in{\mc W}$ and $a\in\CC\setminus\{0\}$.
\end{lemma}

\begin{proof} The equivalences (i)$\Leftrightarrow$(ii)$\Leftrightarrow$(iii) follow from elementary combinatorics of Weyl groups. The equivalences (iii)$\Leftrightarrow$(iv)$\Leftrightarrow$(v) are deduced from Lemma \ref{Lem EdinRootRestric}.
\end{proof}

\begin{theorem}\label{Theo Regular embedding}
Let $\wt{\lambda}\in\wt{\mc P}$ be a regular weight. Denote $\lambda=\iota^*(\wt{\lambda})$ and $\wt{w}=\wt{w}_{\wt{\lambda}}$. Then the nonvanishing of
$$
\pi^{\wt{\lambda}}:H(\wt{X},\wt{\mc O}(\wt{\lambda})) \lw H(X,{\mc O}(\lambda))
$$
is equivalent to each of the conditions {\rm (i)-(v)} in Lemma \ref{Lem HarmRestrictionForRegular}.

Moreover, if $\pi^{\wt{\lambda}}\ne 0$, then

{\rm (a)} for the element $w$ in condition {\rm (v)} we have $w=\wt{w}=w_{\lambda}$;

{\rm (b)} we have $\iota^*(\wt{w}\cdot\wt{\lambda})=w\cdot\lambda$ and the cohomological component given by $(\pi^{\wt{\lambda}})^*:V(w\cdot\lambda)\hookrightarrow\wt{V}(\wt{w}\cdot\wt{\lambda})$ is the highest component, i.e. generated by a $\wt{\mk b}$-highest weight vector.
\end{theorem}

\begin{proof} The necessity of condition (iv) for the nonvanishing of $\pi^{\wt{\lambda}}$ follows from Theorem \ref{Theo Necessity}. Let us prove sufficiency. The conditions of the lemma yield $\wt{w}\in{\mc W}$. Hence the conjugation by $\wt{w}$ preserves $G$ as a subgroup of $\wt{G}$, and acts on $G$ as a Weyl group element $w\in{\mc W}$ preserving $H$. In particular, $\wt{w}^{-1}B\wt{w}$ is a Borel subgroup of $G$ contained in the Borel subgroup $\wt{w}^{-1}\wt{B}\wt{w}$ of $\wt{G}$. It follows that the extreme weight vector $\wt{v}^{\wt{\lambda}+\langle\Phi_{\wt{w}}\rangle}\in \wt{V}(\wt{w}\cdot\wt{\lambda})$ is an eigenvector for $\wt{w}^{-1}B\wt{w}$, and hence is an extreme weight vector for some $G$-submodule of $\wt{V}(\wt{w}\cdot\wt{\lambda})$. Since (v) implies $\iota^*\langle\Phi_{\wt{w}}\rangle=\langle\Phi_w\rangle$, the ${\mk h}$-weight of $\wt{v}^{\wt{\lambda}+\langle\Phi_{\wt{w}}\rangle}$ is $\iota^*(\wt{\lambda}+\langle\Phi_{\wt{w}}\rangle)=\lambda+\langle\Phi_w\rangle$. We can conclude that ${\mk U}(\mk g)\wt{v}^{\wt{\lambda}+\langle\Phi_{\wt{w}}\rangle}$ is an irreducible $G$-module with $B$-highest weight $w(\lambda+\langle\Phi_w\rangle)=w\cdot\lambda$, that is ${\mk U}(\mk g)\wt{v}^{\wt{\lambda}+\langle\Phi_{\wt{w}}\rangle}\cong V(w\cdot\lambda)$. Thus condition (ii) of Theorem \ref{Theo Necessity} is fulfilled and the map $\varpi^{\wt{\lambda}}$ (and consequently $\pi^{\wt{\lambda}}$) is nonzero. A particular harmonic cocycle whose pullback remains harmonic is
$$
\wt{e}^*_{-\Phi_{\wt{w}}}\otimes \wt{f}_{-\wt{\lambda}-\langle\Phi_{\wt{w}}\rangle} \otimes \wt{v}^{\wt{\lambda}+\langle\Phi_{\wt{w}}\rangle} \in C(\wt{\mk n},\wt{V}(\wt{w}\cdot\wt{\lambda})^*\otimes \wt{V}(\wt{w}\cdot\wt{\lambda}))^{\wt{\lambda}} \;.
$$

To prove the second statement of the theorem, suppose $\pi^{\wt{\lambda}}\ne 0$. Then the condition in Lemma \ref{Lem HarmRestrictionForRegular} hold and from the first part of the proof we see that $w=\wt{w}=w_\lambda$, so (a) holds. To prove (b) note that we have $w(\iota^*(\wt{\nu}))=\iota^*(\wt{w}(\wt{\nu}))$ for $\wt{\nu}\in\wt{\mc P}$. We compute
$$
w\cdot\lambda = w(\lambda+\langle\Phi_w\rangle) = w(\iota*(\wt{\lambda}+\langle\Phi_{\wt{w}}\rangle)) = \iota^*(\wt{w}(\wt{\lambda}+\langle\Phi_{\wt{w}}\rangle)) = \iota^*(\wt{w}\cdot\wt{\lambda}) \;.
$$
Finally, observe that $w$ acting on $\wt{V}(\wt{w}\cdot\wt{\lambda})$ sends $\wt{v}^{\wt{\lambda}+\langle\Phi_{\wt{w}}\rangle}$ to $\wt{v}^{\wt{w}\cdot\wt{\lambda}}$. Hence the highest weight vector belongs to the cohomological component ${\mk U}(\mk g)\wt{v}^{\wt{\lambda}+\langle\Phi_{\wt{w}}\rangle}$.
\end{proof}

\begin{coro}\label{Coro Reg C=C_0}
For a regular embedding $\varphi:G/B\hookrightarrow\wt{G}/\wt{B}$ we have ${\mc C}(\varphi)={\mc C}_0(\varphi)$, i.e. all cohomological components can be obtained from pullbacks in cohomological degree 0. (See (\ref{For C in intro}) and Remark \ref{Zab H^0->H^0} for the notation.)
\end{coro}

\subsection{Diagonal embeddings}\label{Sec Diagonal}

In this section we present a new proof of a theorem of Dimitrov and Roth, cf. \cite{Dimitrov-Roth long}, which answers problem I for diagonal embeddings. Our proof is derived from the general Theorem \ref{Theo Necessity}. The specific feature here is that condition (i) of Theorem \ref{Theo Necessity} is necessary and sufficient for the nonvanishing of $\pi^{\wt{\lambda}}$. The fact that condition (ii) can be dropped is deduced from a theorem of Kumar and Mathieu (cf. \cite{Kumar-PRV} and \cite{Mathieu} respectively), known as the PRV conjecture.

Consider here the diagonal inclusion of ${\mk g}$ into $\wt{\mk g}={\mk g}\oplus{\mk g}$,
$$
\iota: {\mk g} \hookrightarrow {\mk g}\oplus{\mk g} \quad,\quad \iota(x)=(x,x) \;.
$$
In this case we also have an inclusion on the level of simply connected Lie groups $\phi:G\hookrightarrow G\times G=\wt{G}$. Furthermore, any Borel subgroup $B\subset G$ is contained in a unique Borel subgroup of $\wt{G}$, namely $\wt{B}=B\times B$. We now assume that $B$ and $\wt{B}=B\times B$ are fixed. This results in a diagonal embedding of the flag manifold $X=G/B$ into $\wt{X}=X\times X$:
$$
\varphi : X \hookrightarrow X\times X \quad,\quad \varphi(x)=(x,x) \;.
$$
Let $K\subset G$ be a maximal compact subgroup. Put $T=K\cap B$ and $H=C_G(T)$; these are maximal toral subgroups of $K$ and $G$ respectively. Then $\wt{K}=K\times K$, $\wt{T}=T\times T$, and $\wt{H}=H\times H$ are respectively a maximal compact subgroup of $\wt{G}$, a maximal torus in $\wt{K}$, and a maximal complex torus in $\wt{G}$. These subgroups also satisfy the conditions we need, namely $K=G\cap\wt{K}$, $T=K\cap\wt{T}$, $H=G\cap\wt{H}$. With this all choices required for our work, as set in the beginning of Section \ref{Sec Maps of Flags}, are made and we can use the relevant notation.

The weight lattice of $\wt{G}$ is $\wt{\mc P}={\mc P}\oplus {\mc P}$ and the pullback $\iota^*:{\mc P}\oplus{\mc P}\lw{\mc P}$ is given by
$$
\iota^*(\lambda_1,\lambda_2)=\lambda_1+\lambda_2 \;.
$$
The Weyl group of $\wt{G}$ is the direct product $\wt{\mc W}={\mc W}\times {\mc W}$; the root system is $\wt{\Delta}=(\Delta\times\{0\})\cup(\{0\}\times\Delta)$. The inversion set of a given element $\wt{w}=(w_1,w_2)\in\wt{\mc W}$ has the form
$$
\Phi_{\wt{w}}=(\Phi_{w_1}\times\{0\})\cup(\{0\}\times\Phi_{w_2}) \;.
$$
This implies $\wt{l}(\wt{w})=l(w_1)+l(w_2)$. The lemma below follows immediately from the definitions.

\begin{lemma}
Let $\wt{\lambda}=(\lambda_1,\lambda_2)\in\wt{\mc P}$. Then:

{\rm (i)} $\wt{\lambda}$ is a regular weight in $\wt{\mc P}$ if and only if both $\lambda_1$ and $\lambda_2$ are regular weights in ${\mc P}$.

{\rm (ii)} If $\wt{\lambda}$ is regular weight, then $\wt{w}_{\wt{\lambda}}=(w_{\lambda_1},w_{\lambda_2})$.
\end{lemma}

Now, let $\wt{\lambda}=(\lambda_1,\lambda_2)\in\wt{\mc P}$ be a regular weight. K${\rm\ddot{u}}$nneth's formula implies that
$$
H^{\wt{l}(\wt{\lambda})}(\wt{X},\wt{\mc O}(\wt{\lambda})) \cong H^{l(\lambda_1)}(X,{\mc O}(\lambda_1)) \otimes H^{l(\lambda_2)}(X,{\mc O}(\lambda_2)) \;.
$$
Put $\lambda=\iota^*(\wt{\lambda})=\lambda_1+\lambda_2\in{\mc P}$. The pullback $\pi^{\wt{\lambda}}$ can be interpreted as the cup product map
$$
\pi^{\wt{\lambda}} : H^{l(\lambda_1)}(X,{\mc O}(\lambda_1)) \otimes H^{l(\lambda_2)}(X,{\mc O}(\lambda_2)) \lw H^{l(\lambda_1)+l(\lambda_2)}(X,{\mc O}(\lambda_1+\lambda_2)) \;.
$$
The reciprocity law described in Section \ref{Sec Cohomologies} allows us to translate this map to a pullback $\varpi^{\wt{\lambda}}$ in Lie algebra cohomology. The latter map is expressed as a wedge product,
\begin{gather*}
\varpi^{\wt{\lambda}} : H({\mk n},{\mc F}(K)\otimes\CC_{-\lambda_1})\otimes H({\mk n},{\mc F}(K)\otimes\CC_{-\lambda_2}) \; \lw \; H({\mk n},{\mc F}(K)\otimes\CC_{-\lambda}) \quad \\
\quad\quad (e^*_{-\Phi_1}\otimes F_1 \otimes z_1) \otimes (e^*_{-\Phi_2}\otimes F_2 \otimes z_2) \; \lo \; (e^*_{-\Phi_1}\wedge e^*_{-\Phi_2}) \otimes (F_1 \otimes F_2) \otimes (z_1\otimes z_2)
\end{gather*}

\begin{theorem}\label{Theo diagonal} {\rm (Dimitrov-Roth, \cite{Dimitrov-Roth long})}
Let $\wt{\lambda}=(\lambda_1,\lambda_2)\in\wt{\mc P}$, and $\lambda=\lambda_1+\lambda_2\in{\mc P}$. Suppose the weights $\lambda_1,\lambda_2,\lambda$ are regular, and put $w_1=w_{\lambda_1}$, $w_2=w_{\lambda_2}$, $w=w_{\lambda}$. Then the following are equivalent:

{\rm (i)} $\pi^{\wt{\lambda}}\ne 0$.

{\rm (ii)} $\varphi_o^*(\wt{e}^*_{-\Phi_{\wt{w}}}) = e^*_{-\Phi_w}$.

{\rm (iii)} $\Phi_{w_1}\sqcup\Phi_{w_2}=\Phi_w$.
\end{theorem}

\begin{proof} The equivalence of (ii) and (iii) is obvious. The implication (i)$\Lw$(ii) was established in the general case in Theorem \ref{Theo Necessity}. To show that (iii) implies (i), it suffices to show that condition (ii) of Theorem \ref{Theo Necessity} is also satisfied. We invoke Theorem 2.10 in \cite{Kumar-PRV}, which states that for any $\mu,\nu\in{\mc P}^+$ and $\sigma\in{\mc W}$, the ${\mk g}$-module $V$ with extreme weight $\mu+\sigma(\nu)$ occurs with multiplicity exactly one in the submodule ${\mk U}(\mk g)(v^{\mu}\otimes v^{\sigma(\nu)})$ of $V(\mu)\otimes V(\nu)$. Now take $\mu=w_1\cdot\lambda_1$, $\nu=w_2\cdot\lambda_2$, and $\sigma=w_1w_2^{-1}$. Then $V=V(w\cdot\lambda)$ and this module has $w^{-1}(w\cdot\lambda)$ as an extreme weight. Using condition (iii) we get
\begin{gather*}
w^{-1}(w\cdot\lambda)=\lambda+\rho-w^{-1}(\rho)=\lambda+\langle\Phi_w\rangle=\lambda_1+\lambda_2+\langle\Phi_{w_1}\rangle+\langle\Phi_{w_2}\rangle\\
=w_1^{-1}(w_1\cdot\lambda_1)+w_2^{-1}(w_2\cdot\lambda_2) =w_1^{-1}(\mu)+w_2^{-1}(\nu)\;.
\end{gather*}
We have ${\mk U}(\mk g)(v^{\mu}\otimes v^{\sigma(\nu)})={\mk U}(\mk g)(v^{w_1^{-1}(\mu)}\otimes v^{w_2^{-1}(\nu)})$ since the diagonal action of the Weyl group is compatible with the ${\mk g}$-module structure of $V(\mu)\otimes V(\nu)$. Hence $V(w\cdot\lambda)\subset{\mk U}(\mk g)(v^{w_1^{-1}(\mu)}\otimes v^{w_2^{-1}(\nu)})$. This completes the proof.
\end{proof}

Dimitrov and Roth studied various properties of cohomological components for diagonal embeddings. In particular, they answered problem III and proposed a conjectural answer to problem II, supported by a proof for classical groups as well as for generic weights of arbitrary semisimple groups. (For the missing definitions see \cite{Dimitrov-Roth long}):

\noindent (II) Conjecture of Dimitrov and Roth: The cohomological components of $V(\mu_1)\otimes V(\mu_2)$ are exactly the generalized PRV components of stable multiplicity one.

\noindent (III) The set ${\mc C}'=\{(\mu,\mu_1,\mu_2)\in({\mc P}^+)^3: V(\mu)^*\subset V(\mu_1)\otimes V(\mu_2)\; {\rm cohomological}\}$ equals the union of all regular walls of codimension $\ell=rank{\mk g}$ of the Littlewood-Richardson cone.

Here a regular wall means a wall intersecting the interior of the tripple dominant chamber. The regular walls of codimension $\ell$ wall are exactly the monoids ${\mc C}'_{w,w_1,w_2}$ for $w,w_1,w_2$ satisfying (iii), defined analogously to our ${\mc C}_{w,\wt{w}}$. (Notice that there is one dualization distinguishing our ${\mc C}$ from ${\mc C}'$.)

\subsection{Homogeneous rational curves}\label{Sec RationalCurves}

In this section, we take ${\mk g}=\mk{sl}_2$ included in a semisimple Lie algebra $\wt{\mk g}$ and study the resulting equivariant embeddings of $X=\PP^1$ into a higher dimensional flag manifold $\wt{X}$. First we consider the special case when ${\mk g}$ is a principal three-dimensional subalgebra of $\wt{\mk g}$ (the definition is give below). Then we use this case and the results on regular embeddings from Section \ref{Sec Regular} to deal with a more general situation.

\subsubsection{Principal rational curves}

We start by recalling, following Kostant \cite{Kostant-PTDS}, the definition as some basic properties of principal three dimensional subalgebras of semisimple Lie algebras. After that, we study the resulting embeddings of flag manifolds and give a complete classification of the cohomological components of modules.

The nilpotent elements in the semisimple complex Lie algebra $\wt{\mk g}$ are divided into conjugacy classes under the action of the adjoint group. The largest class is the one of principal nilpotent elements, which is characterized by the property that each representative $\varepsilon$ is contained in a unique maximal subalgebra $\wt{\mk n}$ consisting of nilpotent elements, $\varepsilon\in\wt{\mk n}\subset\wt{\mk g}$. The normalizer of $\wt{\mk n}$ is a Borel subalgebra $\wt{\mk b}$ of $\wt{\mk g}$. Let $\wt{\mk h}\subset\wt{\mk b}$ be any Cartan subalgebra, to which we associate roots and root vectors as usual. Let $\wt{\alpha}_1,...,\wt{\alpha}_\ell$ be the simple roots associated to $\wt{\mk b}$. Then the principal nilpotent elements contained in $\wt{\mk n}$ are alternatively characterized by the fact that in the expression
$$
\varepsilon=\sum\limits_{\wt{\alpha}\in\wt{\Delta}^+} c_{\wt{\alpha}} \wt{e}_{\wt{\alpha}}
$$
one has $c_{\wt{\alpha}_j}\ne 0$ for all $j=1,2,...,\ell$. (In $\mk{sl}_{\ell+1}$ the principal nilpotent elements are those, whose Jordan form consists of a single Jordan cell with zero eigenvalue.)

A principal three-dimensional subalgebra ${\mk g}$ of $\wt{\mk g}$ is a simple three-dimensional subalgebra containing a principal nilpotent element. The semisimple elements of such a subalgebra are regular, in the sense that their centralizers are Cartan subalgebras. Thus a Cartan subalgebra ${\mk h}$ of ${\mk g}$ is contained in a unique Cartan subalgebra $\wt{\mk h}$ of $\wt{\mk g}$. It follows that any given Borel subalgebra ${\mk b}\subset{\mk g}$ is contained in a unique Borel subalgebra $\wt{\mk b}\subset\wt{\mk g}$. Since $G/B\cong\PP^1$, we deduce that a given principal ${\mk g}$ in $\wt{\mk g}$ determines a unique, equivariantly embedded, rational curve in the flag manifold $\wt{X}$ of $\wt{G}$,
$$
\varphi:\PP^1 \hookrightarrow \wt{X} \;.
$$
We call such a curve a principal rational curve in $\wt{X}$.

Let $\alpha$ denote the positive root of ${\mk g}$, and let $\wt{\alpha}_1,...,\wt{\alpha}_\ell$ denote the simple roots of $\wt{\mk g}$. Let $\wt{\omega}_1,...,\wt{\omega}_\ell$ be the fundamental weights, so that $2\wt{\kappa}(\wt{\alpha}_j,\wt{\omega}_k)/\wt{\kappa}(\wt{\alpha}_j,\wt{\alpha}_j)=\delta_{j,k}$.

We want to determine the possible cohomological components of $\wt{\mk g}$-modules obtained via this embedding. The line bundles on $\PP^1$ carry nonzero cohomology in degree at most 1. The pullbacks in degree 0 never vanish (provided the domain is non-zero) and give rise to cohomological components which are exactly the highest components. From now on, we restrict our considerations to line bundles whose cohomology appears in degree 1.

Let $\wt{\lambda}\in\wt{\mc P}$ be a weight of length 1, put $\lambda=\iota^*(\wt{\lambda})$, and consider the pullback
$$
\pi^{\wt{\lambda}} : H^1(\wt{X},\wt{\mc O}(\wt{\lambda})) \lw H^1(\PP^1,{\mc O}(\lambda)) \;.
$$
The Weyl group element $\wt{w}_{\wt{\lambda}}$ is one of the simple reflections in $\wt{\mc W}$. According to the interpretation of the map $\pi^{\wt{\lambda}}$ in terms of Lie algebra cohomology, we need to study the map
$$
\varphi_o^*:\wt{\mk n}^* \lw {\mk n}^* \;.
$$
Since ${\mk n}=\CC e_{\alpha}$, and the element $e_{\alpha}$ is principal nilpotent, we have a nonzero image
$$
\varphi_o^*(\wt{e}^*_{-\wt{\alpha}_j}) = a_{\wt{\alpha}_j} e^*_{-\alpha} \;, j=1,2,...,\ell \;.
$$
Thus condition (i) from Theorem \ref{Theo Necessity} is always satisfied. This implies in particular that
$$
\iota^*(\wt{\alpha}_j)=\alpha=2 \quad,\quad j=1,2,...,\ell \;.
$$
Hence the restriction of weights $\iota^*:\wt{\mc P}\lw{\mc P}$ is easily computable.
In particular, we can deduce that all fundamental weights of $\wt{\mk g}$ are mapped to strictly positive numbers in ${\mc P}=\ZZ$. This fact already follows from the regularity of $h_{\alpha}$, but now we can be more precise.

\begin{lemma}\label{Lem Princ-FundWeights}
Let $\wt{\alpha}_j$ be a simple root of a simple ideal ${\mk g}_1$ of $\wt{\mk g}$, and let $\wt{\omega}_j$ be the corresponding fundamental weight. Then one of following possibilities occurs:

{\rm (i)} If ${\mk g}_1\cong \mk{sl}_2$, then $\iota^*(\wt{\omega}_j)=1$.

{\rm (ii)} If ${\mk g}_1\cong \mk{sl}_3$, then $\iota^*(\wt{\omega}_j)=2$.

{\rm (iii)} If ${\mk g}_1$ is not isomorphic to $\mk{sl}_2$ or $\mk{sl}_3$, then $\iota^*(\wt{\omega}_j)\geq 3$.
\end{lemma}

\begin{proof}
The result is obtained by direct computations using the tables at the end of \cite{Bourbaki-Lie-2}.
\end{proof}

We can now deduce the following theorem, which solves problems I, II, III for principal rational curves.

\begin{theorem}\label{Theo principal}
Let $\wt{\lambda}\in\wt{\mc P}$ be a weight of length 1, with $\wt{w}_{\wt{\lambda}}=\wt{s}_{\wt{\alpha}_j}$. Then $\wt{s}_{\wt{\alpha}_j}\cdot\wt{\lambda}=\sum_{k=1}^\ell c_k\wt{\omega}_k$ with $c_k\geq 0$. Let ${\mk g}_1$ be the simple ideal of $\wt{\mk g}$ to whose root system $\wt{\alpha}_j$ belongs. Let $\lambda=\iota^*(\wt{\lambda})$. Then $\pi^{\wt{\lambda}}:H^1(\wt{X},\wt{\mc O}(\wt{\lambda}))\rw H^1(\PP^1,{\mc O}(\lambda))$ is a nonzero map if and only if one of the following alternatives holds:

{\rm (i)} ${\mk g}_1\cong \mk{sl}_2$ and $c_j\geq \sum_{k\ne j} c_k\iota^*(\wt{\omega}_k)$. In this case we have a cohomological component $V(-\lambda-2)\subset \wt{V}(\wt{s}_{\wt{\alpha}_j}\cdot\wt{\lambda})$ generated by a highest weight vector; $\lambda$ can take any value in $\ZZ_{\leq-2}$ provided the coefficient $c_j$ is sufficiently large.

{\rm (ii)} ${\mk g}_1\cong \mk{sl}_3$ and $\wt{\lambda}=\wt{s}_{\wt{\alpha}_j}\cdot (m\wt{\omega}_j)$ where $m$ is an even positive integer. In this case $\lambda=-2$ and the cohomological component is the one dimensional space of $G$-invariants in $\wt{V}(m\wt{\omega}_j)$.

{\rm (iii)} $\wt{\lambda}=-\wt{\alpha}_j$. In this case $\lambda=-2$ and $\pi^{\wt{\lambda}}$ is a cohomological morphism from the trivial $\wt{G}$-module to the trivial $G$-module.
\end{theorem}

\begin{proof}
We need to check when $\lambda$ has length 1 (i.e. $\lambda\leq -2$). If it does, then, as noted above, condition (i) of Theorem \ref{Theo Necessity} is always satisfied and it remains to check condition (ii). We have
\begin{gather}\label{For Comput in princ-sl_2}
\begin{array}{rl}
\lambda & = \iota^*(\wt{\lambda}) \\
        & = \iota^*(\wt{s}_{\wt{\alpha}_j}\cdot\sum_k c_k\wt{\omega}_k) \\
        & = \iota^*(\wt{s}_{\wt{\alpha}_j}(\sum_k c_k\wt{\omega}_k)-\wt{\alpha}_j) \\
        & = \iota^*(\sum_k c_k\wt{\omega}_k-a_j\wt{\alpha}_j-\wt{\alpha}_j) \\
        & = \sum_k c_k\iota^*(\wt{\omega}_k)-(a_j+1)2 \;.
\end{array}
\end{gather}
It follows that $\lambda\leq-2$ if and only if $c_j(\iota^*(\wt{\omega}_j)-2)+\sum_{k\ne j} c_k\iota^*(\wt{\omega}_j) \leq 0$. Now we refer to Lemma \ref{Lem Princ-FundWeights}, and we get three possible cases.

First, if ${\mk g}_1\cong\mk{sl}_2$, then part (i) of the lemma applies and we can invoke Theorem \ref{Theo Regular embedding} to conclude that part (i) of the present theorem holds.

Second, if ${\mk g}_1\cong\mk{sl}_3$, then we see that $\lambda\leq-2$ if and only if $c_k=0$ for $k\ne j$. So we must have $\wt{\lambda}=\wt{s}_{\wt{\alpha}_j}\cdot (m\wt{\omega}_j)$. Then we get $\lambda=-2$ and $H^1(\PP^1,{\mc O}(\lambda))=\CC$. Notice that the principal three-dimensional subalgebra of $\mk{sl}_3$ is in fact the image of $\mk{sl}_2$ under the adjoint representation. Depending on the way in which the fundamental weights are ordered, we get that $H^1(\wt{X},{\mc O}(\wt{\lambda}))$ is isomorphic to either $S^m(\mk{sl}_2)$ or $S^m(\mk{sl}_2^*)$. The cohomological component, if it exists, is contained in the invariants of $S^m(\mk{sl}_2)$ or $S^m(\mk{sl}_2^*)$. On the other hand, the $\mk{sl}_2$-invariants in $S^{\cdot}(\mk{sl}_2)$ are generated by the Killing form, and hence occur only in even degrees. (Respectively, in $S^{\cdot}(\mk{sl}_2)$ the invariants are generated by the symmetric tensor associated to the Casimir element.) In particular, if $m$ is odd, there are no invariants and hence the pullback $\pi^{\wt{\lambda}}$ vanishes. To prove that the map is surjective for $m$ even, we need to show that the invariants are contained in the $\mk{sl}_2$-submodule generated by an appropriate extreme weight vector, as required by condition (ii) of Theorem \ref{Theo Necessity}. This can be done directly, by an elementary but somewhat tedious computation, which we shall not reproduce here. Instead one could use the following simple argument, given here for $H^1(\wt{X},{\mc O}(\wt{\lambda}))\cong S^{m}(\mk{sl}_2^*)$, the other case being analogous. The invariants in $S^{m}(\mk{sl}_2^*)$ are spanned by $\kappa^{m/2}$, where $\kappa$ is the Killing form. The extreme weight vector in $S^{m}(\mk{sl}_2)$ taking part in condition (ii) of Theorem \ref{Theo Necessity} is $\wt{v}^{\wt{s}_{\wt{\alpha}_j(m\wt{\omega}_j)}}=h_{\alpha}^m$. Since $\kappa(h_\alpha)\ne 0$ (here $\kappa$ is regarded as a quadratic form) we see that
$$
\varpi^{\wt{s}_{\wt{\alpha}_j}\cdot (m\wt{\omega}_j)}[\wt{e}^*_{-\alpha_j}\otimes (\kappa^{m/2}\otimes h_{\alpha}^m) \otimes\wt{z}] \ne 0 \;,
$$
which implies part (ii) of the present theorem.

Statement (iii) follows directly from Corollary \ref{Coro Triv modules}.

To see that parts (i), (ii) and (iii) account for all possibilities for a nonzero pullback, observe that if ${\mk g}_1$ is neither $\mk{sl}_2$ nor $\mk{sl}_3$, then part (iii) of Lemma \ref{Lem Princ-FundWeights} and the computation (\ref{For Comput in princ-sl_2}) imply that $\lambda$ has length 1 if and only if $c_k=0$ for all $k$. In this case $\wt{\lambda}=-\wt{\alpha}_j$ and $\lambda=-2$, so we are in case (iii) of the theorem.
\end{proof}

\subsubsection{More rational curves}

Combining the results from the preceding sections we can study more general embeddings. Suppose that the homogeneous embedding $\varphi:\PP^1\hookrightarrow\wt{G}/\wt{B}$ factors as the composition of two embeddings $\varphi_1:\PP^1\hookrightarrow G_1/B_1$ and $\varphi_2:G_1/B_1\hookrightarrow\wt{G}/\wt{B}$, so that $\varphi_1$ is a principal embedding and $\varphi_2$ is a regular embedding. Then we can apply Theorem \ref{Theo Regular embedding} to $\varphi_2$ and Theorem \ref{Theo principal} to $\varphi_1$ in order to find the weights $\wt{\lambda}\in\wt{\mc P}$ for which the pullback $\pi^{\wt{\lambda}}:H(\wt{G}/\wt{B},{\mc O}(\wt{\lambda}))\rw H(\PP^1,{\mc O}(\iota^*\wt{\lambda}))$ is nonzero. The above assumption is satisfied if and only if the following property holds:\\

\noindent{\bf Property (R)}: {\it A three dimensional simple subalgebra ${\mk g}$ of a semisimple Lie algebra $\wt{\mk g}$ will be said to satisfy the property {\rm (R)} if there exists a proper semisimple regular subalgebra ${\mk g}_1\subset\wt{\mk g}$, containing ${\mk g}$ as a principal three dimensional subalgebra.}\\

The three dimensional simple subalgebras of a semisimple Lie algebra were classified in a classical work of Dynkin, \cite{Dynkin-SS}. The results in chapter III therein show that (R) holds for any three dimensional simple subalgebra of $\wt{\mk g}$, if $\wt{\mk g}$ is of type $A_\ell,B_\ell,C_\ell,G_2,F_4$ or is a sum of simple algebras of these types. The property (R) is not always satisfied if $\wt{\mk g}$ has simple summands of type $D_\ell$ or $E_\ell$. The exceptions are explicitly classified in \cite{Dynkin-SS}. We refrain from further study of these exceptions. Instead, we present examples where the factorization method applies. In both examples below, we take the same inclusion $\iota:\mk{sl}_2\hookrightarrow\mk{sl}_4$, satisfying the property (R). What differs in the two examples is the choice of Borel subalgebras ${\mk b}\subset\wt{\mk b}$. Thus we get two different embeddings $\PP^1\hookrightarrow SL_4/\wt{B}$, which both factor as desired above. We observe that the respective sets of cohomological components differ substantially. This illustrates on a simple example the dependence of the set of cohomological components on the choice of Borel subalgebras.

\begin{ex} Let ${\mk g}=\mk{sl}_2$, $\wt{\mk g}=\mk{sl}_4$ and $\iota:{\mk g}\hookrightarrow\wt{\mk g}$ be the inclusion given by
\begin{gather*}
\left[\begin{array}{cc} a & b \\ c & d \end{array}\right] \;\stackrel{\iota}{\lo}\;
\left[\begin{array}{cccc} a & b & 0 & 0 \\ c & d & 0 & 0 \\ 0 & 0 & a & b \\ 0 & 0 & c & d \end{array}\right] \;.
\end{gather*}
Let ${\mk h}\subset\mk{sl}_2$ and $\wt{\mk h}\subset\mk{sl}_4$ be the diagonal Cartan subalgebras. It is clear that ${\mk g}_1=\mk{sl}_2\oplus\mk{sl}_2$ is block-diagonally included in $\mk{sl}_4$. Now, we have a diagonal inclusion $\iota_1:{\mk g}\hookrightarrow{\mk g}_1$ and a regular inclusion $\iota_2:{\mk g}_1\hookrightarrow\wt{\mk g}$, so that $\iota=\iota_2\circ\iota_1$.

The restriction of weights is
$$
(\wt{a}_1,\wt{a}_2,\wt{a}_3) \stackrel{\iota_2^*}{\lw} (\wt{a}_1,\wt{a}_3) \stackrel{\iota_1^*}{\lw} \wt{a}_1+\wt{a}_3\;.
$$
(Here the weights are written in coordinates corresponding to the respective bases of fundamental weights.) Consider the upper triangular Borel subalgebra ${\mk b}$ of $\mk{sl}_2$, and as extension take the upper triangular Borel subalgebra $\wt{\mk b}$ of $\wt{sl}_4$. Put ${\mk b}_1={\mk g}_1\cap\wt{\mk b}$. This defines embeddings
$$
\PP^1 \stackrel{\varphi_1}{\hookrightarrow} \PP^1\times\PP^1 \stackrel{\varphi_2}{\hookrightarrow} SL_4/\wt{B} \;.
$$
Let $\wt{\alpha}_1$, $\wt{\alpha}_2$, $\wt{\alpha}_3$ denote the simple roots of $\wt{\mk g}$ with respect to $\wt{\mk b}$, let $\alpha_1,\alpha_2$ denote the simple roots of ${\mk g}_1$ with respect to ${\mk b}_1$, let $\alpha$ denote the simple root of ${\mk g}$ with respect to ${\mk b}$.

We are interested in pullbacks in cohomological degree 1. The relevant inversion sets are singletons of simple roots. Since we are concerned with a regular and a diagonal embedding, we can refer to Theorems \ref{Theo Regular embedding} and \ref{Theo diagonal} respectively, and it is sufficient to consider the restrictions of roots rather than the associated elements $\wt{e}^*_{-\wt{\alpha}}$. We have
$$
\wt{\alpha}_1 \stackrel{\iota_2^*}{\lw} \alpha_1 \stackrel{\iota_1^*}{\lw} \alpha \quad,\quad \iota^*(\wt{\alpha}_2) = 0 \quad,\quad \wt{\alpha}_3 \stackrel{\iota_2^*}{\lw} \alpha_2 \stackrel{\iota_1^*}{\lw} \alpha \;.
$$
Assume $\wt{\lambda}=(\wt{a}_1,\wt{a}_2,\wt{a}_3)\in\wt{\mc P}$ has length 1. Put $\lambda_1=\iota_2^*(\wt{\lambda})$ and $\lambda=\iota^*(\lambda)=\iota_1^*(\lambda_1)$. We have $\pi^{\wt{\lambda}}=\pi_2^{\wt{\lambda}}\circ\pi_1^{\lambda_1}$. Theorems \ref{Theo Regular embedding} and \ref{Theo diagonal} imply that $\pi^{\wt{\lambda}}\ne 0$ if and only if one of the following two cases occurs.

Case 1) $s_{\wt{\alpha}_1}\cdot\wt{\lambda}\in\wt{\mc P}^+$ and $s_{\alpha}\cdot\lambda\in{\mc P}^+$. One can easily compute that in coordinates this translates to
\begin{align*}
\wt{a}_1 & \leq \; -2 \\
\wt{a}_2 & \geq \; 1-\wt{a}_1 \\
0 \leq \wt{a}_3 & \leq \; -2-\wt{a}_1 \;.
\end{align*}
When these inequalities are satisfied we have a cohomological component $V_{\mk{sl}_2}(-\wt{a}_1-\wt{a}_3-2)$ in $V_{\mk{sl}_4}(-\wt{a}_1-2,\wt{a}_1+\wt{a}_2+1,\wt{a}_3)$.

Case 2) $s_{\wt{\alpha}_3}\cdot\wt{\lambda}\in\wt{\mc P}^+$ and $s_{\alpha}\cdot\lambda\in{\mc P}^+$, which translates to
\begin{align*}
0\leq \wt{a}_1 & \leq \; -2-\wt{a}_3 \\
\wt{a}_2 & \geq \; 1-\wt{a}_3 \\
\wt{a}_3 & \leq \; -2 \;.
\end{align*}
When these inequalities are satisfied we have a cohomological component $V_{\mk{sl}_2}(-\wt{a}_1-\wt{a}_3-2)$ in $V_{\mk{sl}_4}(\wt{a}_1,\wt{a}_2+\wt{a}_3+1,-\wt{a}_3-2)$.
\end{ex}

\begin{ex} Here we consider the same inclusion $\mk{sl}_2 \stackrel{\iota_1}{\hookrightarrow} \mk{sl}_2\oplus\mk{sl}_2 \stackrel{\iota_2}{\hookrightarrow} \mk{sl}_4$ as in the previous example, but we associate to it a different embedding of flag manifolds, by choosing different Borel subalgebras. We take ${\mk b}$ and ${\mk b}_1$ as before, but $\wt{\mk b}\subset\mk{sl}_4$ is defined to consist of the matrices of the form
$$
\left[\begin{array}{cccc} * & * & 0 & * \\
                          0 & * & 0 & * \\
                          * & * & * & * \\
                          0 & 0 & 0 & *
\end{array}\right] \;.
$$
Let $\wt{\alpha}_1,\wt{\alpha}_2,\wt{\alpha}_3$ be the simple roots of $\mk{sl}_4$. Then we see that the only simple root of $\mk{sl}_2$ for which $\iota^*(\wt{\alpha}_j)=\alpha$ is $\wt{\alpha}_2$. Let $\wt{\lambda}=(\wt{a}_1,\wt{a}_2,\wt{a}_3)\in\wt{\mc P}$ have length 1. As in the previous example we see that $\pi^{\wt{\lambda}}\ne 0$ if and only if $\wt{s}_{\wt{\alpha}_2}\cdot\wt{\lambda}\in\wt{\mc P}^+$ and $s_{\alpha}\cdot\lambda\in{\mc P}^+$. These conditions yield the following inequalities:
\begin{align*}
\wt{a}_1+\wt{a}_2+1 & \geq \; 0 \\
-\wt{a}_2-2 & \geq \; 0 \\
\wt{a}_2+\wt{a}_3+1 & \geq \; 0 \\
-\wt{a}_1-\wt{a}_3-2 & \geq \; 0 \;.
\end{align*}
It is easy to see that this system is satisfied only for $\wt{\lambda}=-\wt{\alpha}_2=\wt{s}_{\wt{\alpha}_2}\cdot 0$. We can conclude that there is a single weight giving a nonzero cohomological pullback in degree 1, and this pullback is
$$
\pi^{-\wt{\alpha}_2}:\CC \lw \CC \;.
$$
\end{ex}

\subsection{The adjoint representation and invariant polynomials}\label{Sec Invariants}

Let ${\mk g}$ be a semisimple complex Lie algebra and $G$ be the associated connected simply connected complex Lie group. Let $\CC[\mk g]$ denote the coordinate ring of ${\mk g}$. It is a well-known classical result that the ad-invariant polynomials on ${\mk g}$ form a $\CC$-algebra which is isomorphic to a polynomial algebra on $\ell$ variables, where $\ell=rank(\mk g)$, see e.g. \cite{Varadarajan-Invar}. More precisely, there exist homogeneous polynomials $p_1,...,p_\ell\in\CC[\mk g]$ such that
$$
{\mk I} = \CC[\mk g]^G = \CC[p_1,...,p_\ell] \;.
$$
The polynomials $p_1,...,p_j$ are not uniquely determined, but their degrees are. Let $d_j={\rm deg}p_j$. These integers are numerical invariants of ${\mk g}$ and are important in various contexts, see e.g. \cite{Kostant-PTDS}. The numbers $d_1,...,d_j$ are all distinct if ${\mk g}$ is simple. For a semisimple ${\mk g}$ one obtains the totality of degrees for the simple factors, so some numbers may occur with multiplicity. Here is the list of the sets $D=\{d_1,...,d_\ell\}$ for the simple Lie algebras:
\begin{gather*}
\begin{array}{ll}
D(A_\ell) = \{2,3,...,\ell+1\}         & ,\qquad D(G_2) = \{ 2,6 \} \;, \\
D(B_\ell) = \{ 2,4,...,2\ell \}        & ,\qquad D(F_4) = \{ 2,6,8,12 \} \;, \\
D(C_\ell) = \{ 2,4,...,2\ell \}        & ,\qquad D(E_6) = \{ 2,5,6,8,9,12 \} \;, \\
D(D_\ell) = \{ 2,4,...,2\ell-2,\ell \} & ,\qquad D(E_7) = \{ 2,6,8,10,12,14,18 \} \;, \\
                                       & \;\qquad D(E_8) = \{ 2,8,12,14,18,20,24,30 \} \;.
\end{array}
\end{gather*}

Note that the adjoint representation is self-dual so that $\CC[\mk g]=S({\mk g}^*)\cong S(\mk g)$ as ${\mk g}$-modules, and an explicit isomorphism can be obtained using the Killing form. We choose to work with $S(\mk g)$ instead of $S({\mk g}^*)$, so we shall think of $p_1,...,p_\ell$ as elements of $S(\mk g)$ from now on. It is not hard to determine the freedom available for the choice of each generator $p_j$. Obviously scaling is possible and, for types $A_1$ and $A_2$, this is the only freedom. For all other types there are other possible alterations. For instance, for type $A_3$, i.e. ${\mk g}=\mk{sl}_4$, we have $d_1=2, d_2=3, d_3=4$. The degree components of the first two generators are one dimensional, i.e. ${\rm dim}S^2(\mk g)^G={\rm dim}S^3(\mk g)^G=1$. Thus $p_1$ and $p_2$ are determined up to scaling. However, ${\rm dim}S^4(\mk g)^G = 2$ as this space contains $p_1^2$ along with a third generator of ${\mk I}$. Thus $p_3$ is determined up to a summand proportional to $p_1^2$. The general pattern is analogous.

The goal in this section is to show that there exist generators for ${\mk J}$, which may be realized as cohomological components in the spaces $S^{d_1}(\mk g),...,S^{d_\ell}(\mk g)$ respectively. Of course, we need to specify a suitable equivariant embedding of flag manifolds.

Set $\wt{\mk g}=\mk{sl}(\mk g)$ and let $\iota={\rm ad}:{\mk g}\hookrightarrow\wt{\mk g}$ be the inclusion given by the adjoint representation. Note that the degree components $S^d(\mk g)$ are finite dimensional irreducible $\wt{\mk g}$-modules, the symmetric powers of the natural representation, and their highest weights lie along the first fundamental ray of the dominant Weyl chamber. Then $\CC p_j$ is an irreducible ${\mk g}$-submodule in $S^{d_j}(\mk g)$.

\begin{theorem}\label{Theo Invariants}
There exist Borel subalgebras ${\mk b}\subset{\mk g}$ and $\wt{\mk b}\subset\wt{\mk g}$ with ${\mk b}\subset\wt{\mk b}$ such that, for a suitable choice of $p_j$, the ${\mk g}$-submodule $\CC p_j\subset S^{d_j}(\mk g)$ is a cohomological component relative to the embedding of complete flag manifolds $G/B\hookrightarrow\wt{G}/\wt{B}$.

More precisely, the pullback $\pi^{\wt{\lambda}}$ realizing this component is in cohomology of highest possible degree, which is $r={\rm dim}(G/B)=\#\Delta^+$, the restricted weight is $\lambda=\iota^*\wt{\lambda}=\langle\Delta^+\rangle$, and the resulting sheaf ${\mc O}(\lambda)$ is the canonical sheaf on $G/B$.
\end{theorem}

\begin{proof}
We start with the definition of the nested Borel subalgebras. We fix a Cartan and a Borel subalgebras ${\mk h}\subset{\mk b}\subset{\mk g}$, and refer to Section \ref{Sec Preliminar} for the notions and notation related to such a pair. The choice of ${\mk h}$ and ${\mk b}$ is not restricted since all such pairs are conjugate under $G$. The subtlety is to find a suitable $\wt{\mk b}$. In fact we shall consider a family of Borel subalgebras of $\wt{\mk g}$ and choose the appropriate member at the end of the proof. The Borel subalgebras of $\wt{\mk g}$ are in a bijection with the complete flags in the natural representation of $\wt{\mk g}$. In turn the natural representation of $\wt{\mk g}$ coincides, as a vector space, with ${\mk g}$. The condition that ${\mk b}\subset\wt{\mk b}$ means that the flag ${\mc Fl}_{\wt{\mk b}}$ corresponding to $\wt{\mk b}$ must be stable under the (adjoint) action of ${\mk b}$. We shall consider a family of flags in ${\mk g}$ parametrized by the manifold ${\mc Fl}({\mk h})$ of complete flag in the Cartan subalgebra ${\mk h}$. In our result we shall care mainly about the projection of ${\mc Fl}(\mk h)$ onto the projective space $\PP(\mk h)$ of one-dimensional subspaces of ${\mk h}$.

{\it Definition of the flag ${\mc Fl}(y)$ for $y\in{\mc Fl}(\mk h)$}. We shall define the flag ${\mc Fl}(y)$ by an ordered basis of ${\mk g}$ consisting of root vectors and a basis of ${\mk h}$. The flag ${\mc Fl}(y)$ shall be a completion of the partial flag in ${\mk g}$ given by the triangular decomposition ${\mk g}={\mk m}\oplus{\mk h}\oplus{\mk m}^-$. {\it Fix} {\it any} order $\beta_1,...,\beta_r$ on the set of positive roots $\Delta^+$ such that the corresponding flag in ${\mk m}$ is fixed by the action of ${\mk b}$ (one such order can be obtained by first fixing a partial order according to the heights of the roots and then ordering the sets of equal heights arbitrarily). Let $h_1,...,h_\ell$ be {\it any} ordered basis of ${\mk h}$, and denote the corresponding flag in ${\mk h}$ by $y=y(h_1,...,h_\ell)$. Now consider the ordered basis
\begin{gather}\label{For basis in invar}
e_{\beta_1},...,e_{\beta_r},h_1,...,h_\ell,e_{-\beta_r},...,e_{-\beta_1} \;,
\end{gather}
to be denoted by $v_1,...,v_n$ (with $n=\ell+2r$). Let ${\mc Fl}(y)$ be the corresponding complete flag in ${\mk g}$, and let $\wt{\mk b}=\wt{\mk b}(y)$ be the Borel subalgebra of $\wt{\mk g}$ stabilizing ${\mc Fl}(y)$. It is clear that ${\mk b}\subset\wt{\mk b}$.

The basis (\ref{For basis in invar}) defines a Cartan subalgebra $\wt{\mk h}\subset\wt{\mk g}$ consisting of the elements which are diagonal in this basis, as well as a set of simple roots corresponding to the order. Fix a compact real form ${\mk k}\subset{\mk g}$. Since the adjoint representation is irreducible, there is a unique compact real form $\wt{\mk k}\subset\wt{\mk g}$ containing ${\mk k}$. We now make the additional assumption that the basis $h_1,...,h_\ell$ is orthogonal with respect to the unique Hermitian form on ${\mk g}$ which is ${\mk k}$-invariant. Note that this assumption does not limit the choice of $y\in{\mc Fl}(\mk h)$, because every flag can be obtained from an orthogonal ordered basis. Furthermore, we may assume that the root vectors $e_\alpha$ satisfy the relations (\ref{For OrthRel e_alf}), and with this assumption the basis (\ref{For basis in invar}) is orthogonal. This ensures that we have $\wt{\mk h}=\wt{\mk t}\oplus i\wt{\mk t}$, where $\wt{\mk t}=\wt{\mk k}\cap\wt{\mk b}$. We now have everything in place, ready to apply Theorem \ref{Theo Necessity}.

So far we have fixed a basis for the natural representation of $\wt{\mk g}$, as well as the corresponding Cartan subalgebra $\wt{\mk h}$ consisting of diagonal matrices and Borel subalgebra $\wt{\mk b}$ consisting of upper-triangular matrices. Before proceeding further with the proof let us recall some commonly used notation for the roots and Weyl group elements of $\mk{sl}_n$. For $j,k\in\{1,...,n\}$ let $E_{j,k}$ denote the elementary matrix having 1 at the $(j,k)$-th entry and zeros elsewhere. Then $\{ E_{j,k}:j\ne k \}$ is a complete set of root vectors for $\wt{\mk g}$ with respect to $\wt{\mk h}$. Let $\wt{\alpha}_{j,k}$ denote the root corresponding to $E_{j,k}$ for $j\ne k$. Then the simple roots of $\wt{\mk b}$ are $\wt{\alpha}_{j,j+1}$ for $j\in\{1,...,n-1\}$. Let $\wt{s}_1,...,\wt{s}_{n-1}$ denote the corresponding simple reflections generating $\wt{\mc W}$.

Let us exhibit an element $\wt{w}\in\wt{\mc W}$ such that $\varphi_o^*(\wt{e}^*_{-\Phi_{\wt{w}}}) = a e^*_{-\Delta^+}$. Recall that, since we are considering an embedding of complete flag manifolds, we have $\iota=\varphi_o$. The image of $e_{\Delta^+}$ under $\iota$ can be written as a linear combination of simple tensors:
\begin{gather*}
\iota(e_{\Delta^+}) = \sum\limits_{\wt{\Phi}\subset\wt{\Delta}^+} c_{\wt{\Phi}} \wt{e}_{\wt{\Phi}} \;.
\end{gather*}
We are looking for $\wt{w}$ such that $c_{\Phi_{\wt{w}}} \ne 0$. Notice that, for every $j\in\{1,...,r\}$ we have
$$
\iota(e_{\beta_j})v_{r+1} = [e_{\beta_j},h_1] = -\beta_j(h_1)e_{\beta_j} = -\beta_j(h_1) v_j \;.
$$
Put $\wt{\Phi} = \{ \wt{\alpha}_{j,r+1} : 1\leq j \leq r \} \subset\wt{\Delta}^+$. Then we have
\begin{gather*}
c_{\wt{\Phi}} = \prod\limits_{j=1}^r -\beta_j(h_1) \;.
\end{gather*}
Assume now that $h_1$ is regular, so that $c_{\wt{\Phi}}\ne 0$. The set $\wt{\Phi}$ is in fact the inversion set $\Phi_{\wt{w}}$ of the Weyl group element
$$
\wt{w} = \wt{s}_1\wt{s}_2 \dots \wt{s}_r \;.
$$
Hence we have
\begin{gather}\label{For i for invar}
\varphi_o^*(\wt{e}^*_{-\Phi_{\wt{w}}}) = c_{\wt{\Phi}} e^*_{-\Delta^+} \;.
\end{gather}
The second step is to determine all dominant weights $\wt{\mu}\in\wt{\mc P}^+$ such that $w_o\cdot(\iota^*(\wt{w}^{-1}\cdot\wt{\mu}))\in{\mc P}^+$. In the presence of (\ref{For i for invar}) it is equivalent to working with the linear action of the Weyl groups and considering $w_o(\iota^*(\wt{w}^{-1}(\wt{\mu})))$. For the computations it is more convenient to write the weights of $\wt{\mk g}$ as ``$gl$-weights'', i.e. $\wt{\lambda}=(\wt{\lambda}_1,...,\wt{\lambda}_n)$ (then the dominant ones are those whose components are non-increasing integers). In these terms, the restriction of weights is written as
$$
\iota^*(\wt{\lambda}_1,...,\wt{\lambda}_n) = \wt{\lambda}_1\beta_1 + \dots + \wt{\lambda}_r\beta_r - \wt{\lambda}_{n-r+1}\beta_r - \dots - \wt{\lambda}_n\beta_1 \;.
$$
Now, we take $\wt{\mu}=(\wt{\mu}_1,...,\wt{\mu}_n)\in\wt{\mc P}^+$ and compute
\begin{align*}
\iota^*(\wt{w}^{-1}(\wt{\mu})) & = \iota^* (\wt{\mu}_2,\wt{\mu}_3,...,\wt{\mu}_{r+1},\wt{\mu}_1,\wt{\mu}_{r+2},...,\wt{\mu}_n) \\
                               & = \wt{\mu}_2\beta_1 + \dots + \wt{\mu}_{r+1}\beta_r - \wt{\mu}_{n-r+1}\beta_r - \dots - \wt{\mu}_n\beta_1 \;.
\end{align*}
We are interested to know when $\mu=w_o(\iota^*(\wt{w}^{-1}(\wt{\mu})))$ is dominant. It is dominant if and only if $w_o\mu=\iota^*(\wt{w}^{-1}(\wt{\mu}))$ is anti-dominant; thus we do not need to know the exact form of $w_o$. For $w_o\mu$ to be anti-dominant, it is necessary that its inner product with any strictly dominant weight $\nu$ be non-positive. Let $\nu$ be strictly dominant. We have
\begin{gather*}
(\nu,w_o\mu) = (\wt{\mu}_2-\wt{\mu}_n)(\nu,\beta_1) + \dots + (\wt{\mu}_{r+1}-\wt{\mu}_{n-r+1})(\nu,\beta_r) \;.
\end{gather*}
Since $(\nu,\beta_j)>0$ and $\wt{\mu}_{j+1}-\wt{\mu}_{n-j+1} \geq 0$ for all $j\in\{1,...,r\}$, it follows that $(\nu,w_o\mu)\leq 0$ if and only if $(\nu,w_o\mu)=0$ if and only if $\wt{\mu}_2=\dots=\wt{\mu}_n$ if and only if $\wt{\mu}=k\wt{\omega}_1$. (Here $\wt{\omega}_1$ is the first of the fundamental weights $\wt{\omega}_1,...,\wt{\omega}_{n-1}$ of $\wt{\mk g}$.) In such a case, clearly $\mu=0$. To summarize, for $\wt{\mu}\in\wt{\mc P}^+$, we have
\begin{gather*}
\mu=w_o\iota^*(\wt{w}^{-1}(\wt{\mu})) \in {\mc P}^+ \tst \wt{\mu}=k\wt{\omega}_1 \; for \; some \; k\geq 0 \;,
\end{gather*}
and
\begin{gather*}
\mu=w_o\iota^*(\wt{w}^{-1}(\wt{\mu})) \in {\mc P}^+ \;\Lw\; \mu=0 \;.
\end{gather*}

We can conclude that the only possible cohomological components associated with the fixed embedding $\varphi$ and pair of Weyl group elements $\wt{w},w_o$ are trivial $G$-modules included in symmetric powers of the natural representation of $\wt{G}$, i.e. ad-invariant polynomials on ${\mk g}^*$. Specifically, for $k\geq 0$ and $\wt{\lambda}=\wt{w}^{-1}\cdot(k\wt{\omega}_1)$, we have
$$
\pi^{\wt{\lambda}} : S^k(\mk g)^* \lw \CC \;.
$$
Now we have come to verifying condition (ii) of Theorem \ref{Theo Necessity}. The relevant extreme weight vector $\wt{v}=\wt{v}^{\wt{w}^{-1}(\wt{\mu})}$ in this case equals $\wt{v}=v_{r+1}^k=h_1^k\in S^k(\mk g)$. We are looking for a trivial $G$-submodule of $S^{k}(\mk g)^*$ whose non-zero elements do not vanish at $\wt{v}$ as linear functionals. Notice that, for $f\in S^k(\mk g)^*$, we have
$$
f(\wt{v}) = k! f(h_1) \;,
$$
where on the right hand side we consider $f$ as a polynomial function on ${\mk g}$. Thus $f(\wt{v})\ne 0$ if and only of $f(h_1)\ne 0$. Recall that $h_1$ is already restricted to be outside the root hyperplanes. Now, let $Z\subset{\mk h}$ be the Zariski closed set defined by the vanishing of the roots and all nonconstant ad-invariant polynomials. Then, if we choose $h_1 \notin Z$ we get $\pi^{\wt{\lambda}} \ne 0$ for all $k$ such that $S^k(\mk g)^G\ne 0$. Furthermore, the images of $(\pi^{\wt{\lambda}})^*$ for $k=d_1,...,d_\ell$ generate ${\mk I}$. This completes the proof of the theorem.
\end{proof}

\begin{zab}\label{Zab Invar paramet}
It is evident from the proof of Theorem \ref{Theo Invariants} that the resulting cohomological components vary with $h_1$. Several questions arise. For instance: is it true that ${\rm every}$ set of generators for ${\mk J}$ can be obtained for a suitable $h_1$? Or, since $h_1$ may be chosen so that so the the generators vanish on it and the associated cohomological components generate a proper subalgebra of ${\mk J}$, one could ask: which subalgebras can be obtained in this way? To answer these questions one needs relevant knowledge about the zero-loci of generators of ${\mk J}$.
\end{zab}

\begin{zab}\label{Zab Invar i&ii indep}
The construction described in the proof of Theorem \ref{Theo Invariants} can be used to show that conditions {\rm (i)} and {\rm (ii)} of Theorem \ref{Theo Necessity} are independent in general. Namely, if the element $h_1$ is chosen in the zero locus of $p_2$ but not on a root hyperplane, we obtain a map $\pi^{\wt{w}^{-1}\cdot 2\wt{\omega}_1}$ for which {\rm (i)} is satisfied while {\rm (ii)} isn't. If $h_1$ is chosen on a root hiperplane, but not in the zero locus of $p_2$, then $\pi^{\wt{w}^{-1}\cdot 2\wt{\omega}_1}$ satisfies {\rm (ii)} but not {\rm (i)}.
\end{zab}

\begin{zab}\label{Zab Invar&Monoid}
Along the lines of the preceding remarks we obtain diverse examples for the monoid ${\mc C}_{w,\wt{w}}$ of cohomological pairs of dominant weights studied in Section \ref{Sec PropCohCom} (see also Remark \ref{Zab Semigroups}). Here $w=w_o$ and $\wt{w}$ are as in the proof of the above theorem. We see that ${\mc D}_{w_o,\wt{w}}=\{ (0,m\wt{\omega}_1) : m\in\ZZ_{\geq 0} \}$. Since $S^{1}(\mk g)$ does not contain invariants, $(0,\wt{\omega}_1)\notin {\mc C}_{w_o,\wt{w}}$, and hence ${\mc D}_{w_o,\wt{w}}\setminus {\mc C}_{w_o,\wt{w}}\ne\varnothing$. This complement can be finite or infinite. Indeed, let ${\mk g}=\mk{sl}_3$. If $h_1$ is in a generic position, we obtain ${\mc C}_{w_o,\wt{w}} = {\mc D}_{w_o,\wt{w}}\setminus(0,\wt{\omega}_1)$. If $h_1$ is not on a root hyperplane but in the zero locus of $p_2$, respectively $p_3$, we obtain ${\mc C}_{w_o,\wt{w}}=\{(0,m\wt{\omega}_1) : m\equiv 0\; mod\; 2\}$, respectively ${\mc C}_{w_o,\wt{w}}=\{(0,m\wt{\omega}_1) : m\equiv 0\; mod\; 3\}$. In the last two cases, we observe that the number $k$ from Theorem \ref{Theo C_w,wt(w) monoid}, for which $k{\mc D}_{w_o,\wt{w}}\subset{\mc C}_{w_o,\wt{w}}$, equals respectively $2$ and $3$. Hence this number is not determined by the embedding $\iota$ of Lie algebras, but really depends on the choice of Borel subalgebras, i.e. depends on the embedding $\varphi$ of flag manifolds.
\end{zab}

\section*{Acknowledgements}

The work presented here is part of my thesis written under the supervision of Ivan Dimitrov, to whom I use this opportunity to express my deep gratitude and admiration. I am also grateful to Michael Roth, for his insightful lectures and for many useful discussions. I would like to thank Vera Serganova for pointing out a mistake in an earlier version of the text. The research and exposition were done in the Graduate School of Queen's University, and partially supported by an OGS scholarship. The final revision of the text was done at Ruhr-Universit\"at, Bochum, with the support of the Deutsche Forschungsgemeinschaft, SFB/TR12.

\end{document}